\documentclass[11pt]{amsart}
\pdfoutput=1
\usepackage{graphicx}
\usepackage{wasysym}
\usepackage{amssymb, amsmath, amscd}
\usepackage{float}
\usepackage[mathscr]{euscript}
\usepackage{array}
\usepackage[bookmarks=false]{hyperref}

\newtheorem{theorem}{Theorem}[section]
\newtheorem{corollary}[theorem]{Corollary}
\newtheorem{lemma}[theorem]{Lemma}

\newtheorem{proposition}[theorem]{Proposition}

\newtheorem{question}[theorem]{Question}
\newtheorem{conjecture}[theorem]{Conjecture}
\newtheorem{definition-proposition}[theorem]{Definition-Proposition}
\newtheorem{lemma-notation}[theorem]{Lemma-Notation}

\theoremstyle{definition}
\newtheorem{definition}[theorem]{Definition}

\newtheorem{remark}[theorem]{Remark}

\newcommand{\N}{\mathbb{N}}

\newcommand{\Z}{\mathbb{Z}}

\newcommand{\Q}{\mathbb{Q}}

\newcommand{\R}{\mathbb{R}}
\newcommand{\C}{\mathbb{C}}

\newcommand{\twopartdefotherwise}[3]
{
	\left\{
		\begin{array}{ll}
			#1 & \mbox{if } #2 \\
			#3 & \mbox{\textrm{otherwise.}}
		\end{array}
	\right.
}

\newcommand{\threepartdef}[6]
{
	\left\{
		\begin{array}{lll}
			#1 & \mbox{if } #2 \\
			#3 & \mbox{if } #4 \\
			#5 & \mbox{if } #6
		\end{array}
	\right.
}

\newcommand{\twobytwo}[4]
{
	\begin{pmatrix}
		#1 & #2 \\
		#3 & #4 \\
	\end{pmatrix}
}

\begin{document}
\title{Hurwitz theory of elliptic orbifolds, I}
\author{Philip Engel}
\address[Philip Engel]{University of Georgia}
\thanks{Research partially supported by NSF grant DMS-1502585.}
\email{philip.engel@uga.edu}
\maketitle

\begin{abstract} An {\it elliptic orbifold} is the quotient of an elliptic curve by a finite group. In 2001, Eskin and Okounkov proved that generating functions for the number of branched covers of an elliptic curve with specified ramification are quasimodular forms for $SL_2(\Z)$. In 2006, they generalized this theorem to branched covers of the quotient of an elliptic curve by $\pm 1$, proving quasi-modularity for $\Gamma_0(2)$. We generalize their work to the quotient of an elliptic curve by $\langle \zeta_N\rangle$ for $N=3$, $4$, $6$, proving quasi-modularity for $\Gamma(N)$, and extend their work in the case $N=2$.

It follows that certain generating functions of hexagon, square, and triangle tilings of compact surfaces are quasi-modular forms. These tilings enumerate lattice points in moduli spaces of flat surfaces. We analyze the asymptotics as the number of tiles goes to infinity, providing an algorithm to compute the Masur-Veech volumes of strata of cubic, quartic, and sextic differentials. We conclude a generalization of the Kontsevich-Zorich conjecture---these volumes are polynomial in $\pi$. \end{abstract}

\section{Introduction}

There are seventeen two-dimensional {\it crystallographic} or {\it wallpaper} groups: Symmetry groups of biperiodic tilings of $\C\cong \R^2$. Five of them preserve orientation. These five groups are the symmetries of an infinite planar tiling by the tiles shown in Figure \ref{5tiles}. These wallpaper groups are extensions of a translation group, isomorphic to $\Z^2$, by a rotation group $\langle \zeta_N\rangle$, for some $N\in \{1,2,3,4,6\}$. The ducks on each tile serve to break some or all of the rotational symmetries of the tiling.

The quotient of $\C$ by an orientation preserving wallpaper group is an {\it elliptic orbifold}, so named since the quotient by the translation subgroup is an elliptic curve $E$ which the rotation group further acts on. Note that when $N=3,4,6$, the translation group, up to a complex scalar, is $\Z[\zeta_N]$. So $E$ is uniquely determined.

This paper is concerned with the enumeration of branched covers of $X_N:=E/\langle \zeta_N\rangle$, which are the orbifold curves $\mathbb{P}_{2,2,2,2}$, $\mathbb{P}_{3,3,3}$, $\mathbb{P}_{2,4,4}$, $\mathbb{P}_{2,3,6}$ for $N=2,3,4,6$ respectively. These are exactly the one-dimensional Calabi-Yau orbifolds. Physics predicts that generating functions of curve counts with target $X_N$ are quasi-modular forms. For the Gromov-Witten theory of $X_N$, these results were proven for $N=1$ in \cite{op}, for $N=2$ in \cite{shen}, and for $N=3,4,6$ in \cite{mr}.

In this paper, we focus on {\it Hurwitz theory} of $X_N$, which employs representation theory of the symmetric group to enumerate branched covers with a specified degree and branching profile \cite{hurwitz}. Roughly, the main theorem of this paper, extending the results of Eskin and Okounkov in the $N=1$ \cite{eo1} and $N=2$ \cite{eo2} cases, is:

\begin{figure}\hspace{5pt}
\includegraphics[width=1.8in]{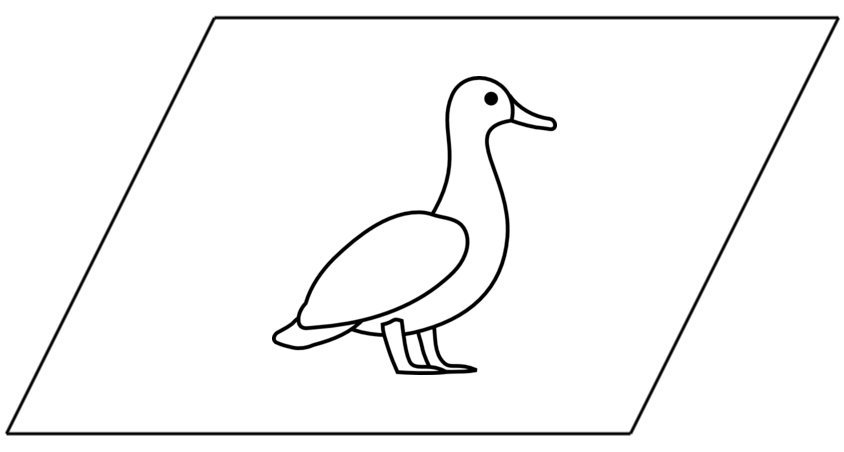}
\includegraphics[width=1.8in]{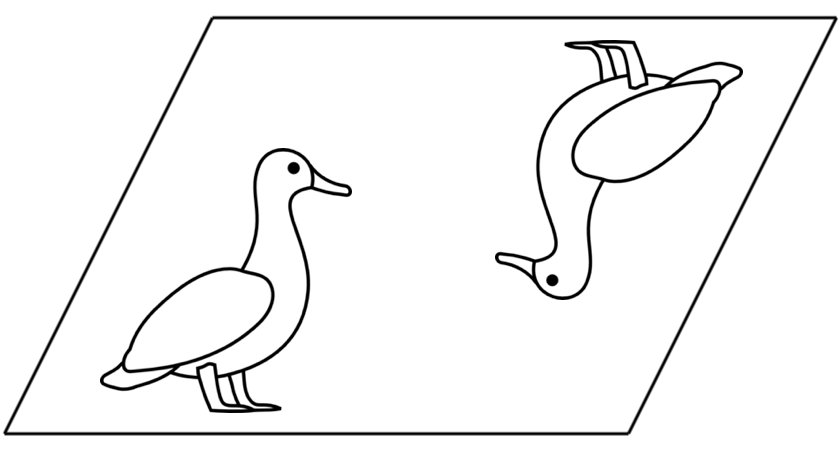} \\ \vspace{3pt}

\includegraphics[width=1.5in]{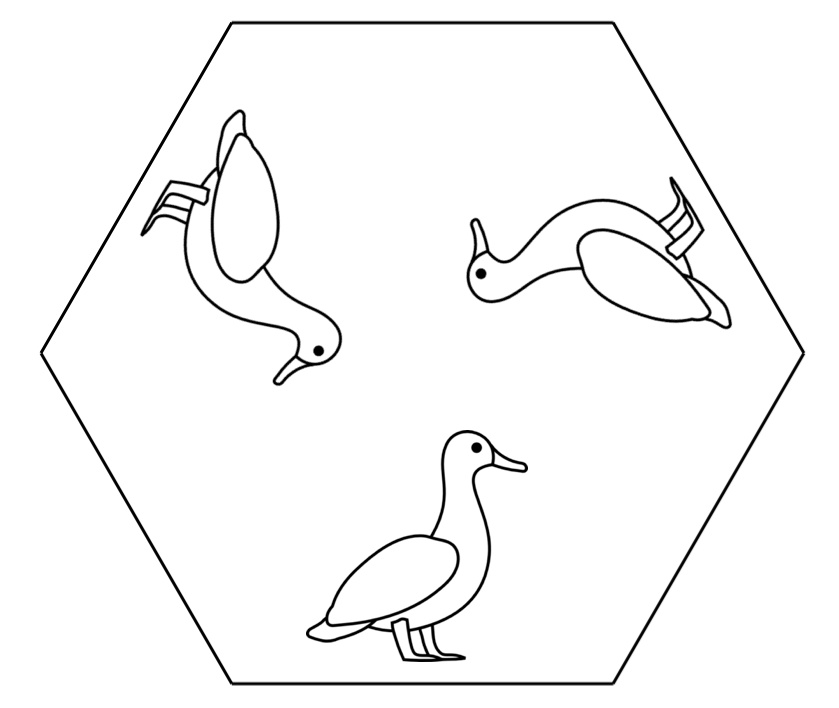}
\includegraphics[width=1.3in]{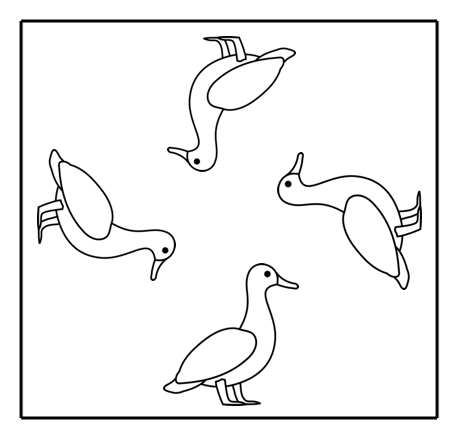}
\includegraphics[width=1.51in]{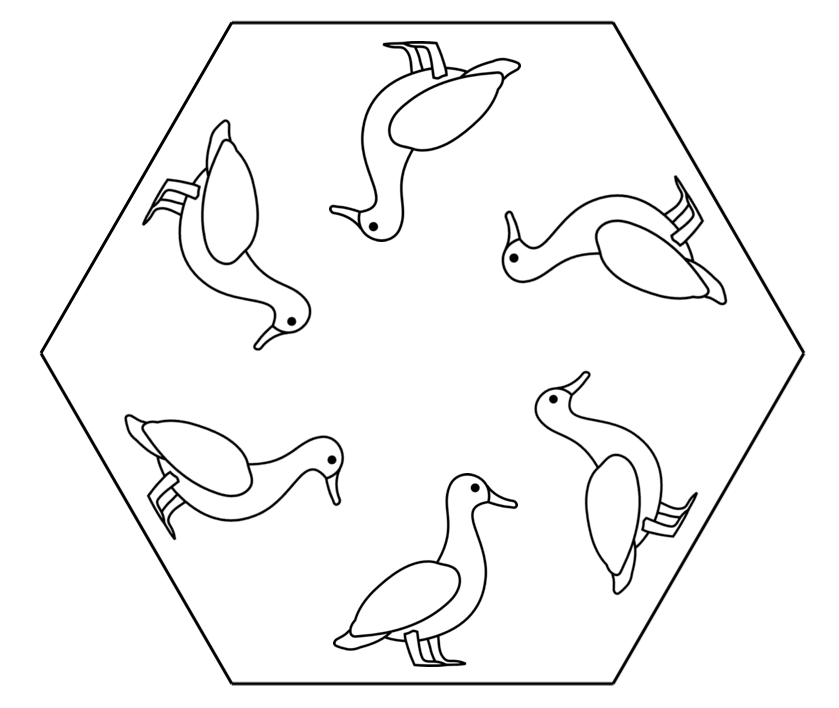}
\caption{Tiles for the five orientation-preserving wallpaper groups}
\label{5tiles}
\end{figure}

\begin{theorem}\label{intro} Let $N=1,2,3,4$, or $6$. Fix a branching profile $D$. Then the generating function $$H_N(D,q)=\!\!\!\!\sum_{\substack{ \textrm{covers }\sigma\,:\,\Sigma\rightarrow X_N\ \\ \textrm{with branching} \\ \textrm{profile }D}} \frac{1}{|{\rm Aut}\,\sigma|}\,q^{\deg(\sigma)/N}$$ is a quasimodular form of mixed weight for $\Gamma(N)$ under the substitution $q=e^{2\pi i\tau}$. Furthermore, $D$ determines an upper bound on the weight.\end{theorem}

In Section 2, we prove that covers $\Sigma\rightarrow X_N$ with appropriate branching profile are naturally in bijection with tilings of $\Sigma$ by the corresponding tile in Figure \ref{5tiles}. Here a {\it tiling} is purely combinatorial---a polyhedral complex built from one of the above polygons, such that edges are glued to edges, with adjacent tiles in the same orientation. As a consequence, one has the following application to combinatorics:

\begin{corollary}\label{intro2} Consider one of the tiles in Figure \ref{5tiles} or the triangle. The generating function for the number of tilings with specified list of non-zero curvatures is a quasimodular form of the appropriate level.\end{corollary}

Here the curvature is the difference between the valence of a vertex with the valence of a tiling of the Euclidean plane by the same tile.

The quasimodularity of the generating function for tilings has application in the study of higher differentials \cite{bcggm}. A {\it $1/N$-translation surface} is a surface with a flat metric away from a finite set with conical singularities, such that the monodromy lies in an order $N$ rotation group. Equivalently, one can consider pairs $(\Sigma,\omega)$ where $\Sigma$ is a Riemann surface and $\omega$ is a meromorphic section of the pluricanonical bundle $K^{\otimes N}$ with poles of order less than $N$. Away from its zeroes and poles, $\omega$ can be locally expressed as $(dz)^N$ for a flat coordinate $z$ well defined up to translation and order $N$ rotation. A {\it stratum} $\mathcal{H}_N(\mu)$ is the moduli space of flat surfaces whose singularities have specified list of curvatures, or equivalently pairs $(\Sigma,\omega)$ such that $\omega$ has specified orders of zeroes and poles $\mu=\{\mu_1,\dots,\mu_n\}$.

Tiled surfaces with appropriate curvatures correspond to special points in $\mathcal{H}_N(\mu)$ by specifying a flat metric on each tile---for instance, declaring each square to be regular of unit side length. The tiled surfaces evenly sample a natural volume form on $\mathcal{H}_N(\mu)$, generalizing the Masur-Veech volume form \cite{masur,veech} on strata of abelian and quadratic differentials.

In the $N=1$ case, Eskin and Okounkov \cite{eo1} proved a conjecture of Kontsevich and Zorich that the volume of any stratum of translation surfaces lies in $\Q\cdot \pi^{2g}$. By analyzing the asymptotics of quasimodular forms of level $\Gamma_1(N)$, Corollary \ref{intro2} allows for the algorithmic computation of the volume of a stratum of cubic, quartic, or sextic differentials. It is tempting to conjecture that for all $N$ and $\mu$, the volume of the stratum is a power of $\pi$ times a rational number. Relying on the recent finiteness result of \cite{ng} we prove the weaker statement:

\begin{theorem}\label{intro3} Let $N=3,4,6$. The volume of the stratum $\mathcal{H}_N(\mu)$ is a polynomial in $\pi$.\end{theorem}

In Section \ref{tilings}, we give the bijection between tilings and branched covers of $X_N$. Then, we review classical results on enumeration of branched covers. The result is a formula for the tiling generating function, and other Hurwitz numbers of elliptic orbifolds, in terms of characters of the symmetric groups.

In Section \ref{quotients}, we review the theory of the $N$-quotients and $N$-core of a partition. We summarize some results in the representation theory of symmetric groups, and introduce the (charge zero subspace of) {\it Fock space} $\mathbb{L}=\bigoplus \C v_\lambda$, an infinite-dimensional vector space with a canonical basis $\{v_\lambda\}$ indexed by partitions of all integers. We describe the action of the Heisenberg algebra on Fock space and the relationship to characters of symmetric groups.

In Section \ref{shiftedsymm}, we manipulate the output of Hurwitz theory into a suitable form. The result of these manipulations is to express $H_N(D,q)$ in terms of certain weighted sums over all partitions $\lambda=\{\lambda_1\geq \dots\geq \lambda_n\}$: $$\sum_{\lambda} {\bf F}(\lambda){\bf w}_{N,\eta}(\lambda)q^{|\lambda|/N}.$$ Here ${\bf F}(\lambda)$ lies in an enlargement $\Lambda_N^*$ of the algebra $\Lambda^*$ of {\it shifted-symmetric polynomials}, i.e.\!\! polynomials which are symmetric in the variables $\lambda_i-i$, and ${\bf w}_{N,\eta}(\lambda)$ is the {\it $\eta$-weight}, a weight on partitions depending on an $N$-core partition $\eta$ which may be expressed simply in terms of the hook lengths of $\lambda$. When $N=2, 3, 4$, or $6$, it is closely related to the Hurwitz theory of the order $N$ elliptic orbifold, but it naturally generalizes to all $N$.

In Section \ref{proof}, we prove for $\eta=\emptyset$ that these weighted sums lie in the ring of quasi-modular forms of level $\Gamma_1(N)$. Closely following \cite{eo2}, we express the above sum in terms of the trace of a product of {\it vertex operators} acting on Fock space. The crux of the proof is to recompute this trace in another canonical basis. This change of basis is an instantiation of the so-called {\it boson-fermion correspondence}. Theorem \ref{intro} follows. 

In Section \ref{volumesec}, we use our results to outline the computation of volumes of moduli spaces of cubic, quartic, and sextic differentials. We conclude Theorem \ref{intro3}. We state some open questions which the results raise. In the appendix, we work through a numerical example to compute a generating function of tiled surfaces and the volume of the stratum of cubic differentials $\mathcal{H}_3(-1,-1,-2,-2)$.

{\bf Acknowledgements:} Many thanks to Andrei Okounkov, Yaim Cooper, Peter Smillie, Eduard Duryev, Yu-Wei Fan, Curtis McMullen, Elise Goujard, Martin M\"oller, Yefeng Shen, and others for their discussion and correspondence. In addition, I thank Georg Oberdieck for pointing out that the integral of a multivariate elliptic function need not be elliptic in the remaining variables.

\section{Surface tilings and Hurwitz theory}\label{tilings}

\subsection{Tilings as branched covers} We now describe the bijection between tilings and branched covers of $X_N$.

\begin{definition} Let $N=1,2,3,4,6$. An {\it $N$-duck tiling} is a compact oriented surface formed from identifying pairs of edges of a collection of disjoint oriented $N$-duck tiles, see Figure \ref{5tiles}. \end{definition}

\begin{remark} For the $3$-duck tile, we require that only one duck stands over a given edge. This will ensure the monodromy of the flat metric described below lies in $\langle \zeta_3\rangle$. \end{remark}

The equivalence relation on tilings is combinatorial: Two tilings are isomorphic if there are bijections between the vertices, edges, and $2$-cells i.e. tiles which preserve the incidences.

Let $\mathcal{T}$ be an $N$-duck tiling of a compact, oriented surface $\Sigma$. Define a flat metric on $\Sigma$, with possible conical singularities at the vertices of $\mathcal{T}$, by declaring each tile to have a flat metric pulled back from the planar embedding shown in Figure \ref{5tiles}, and identifying edges via orientation-preserving isometries. The allowed identifications of edges lie in the group $G:=\langle \zeta_N\rangle\ltimes \C$.

\begin{definition} The {\it curvature} of a vertex of an $N$-duck tiled surface is $$N\left(1-\frac{\theta}{2\pi}\right)$$ where $\theta$ is the cone angle around the vertex. Note that the curvatures of a $6$-duck tiled surface are always even. Similarly, the curvature of a vertex of a triangulation is six minus its valence.  \end{definition}

Define $\mu:=-\kappa+N=\{N-\kappa_1,\,\dots,\,N-\kappa_n\}$. 

\begin{proposition}\label{zero} The $N$-duck tiled surfaces $\Sigma$ with $d$ tiles and curvatures $\kappa$ are in bijection with:
\begin{enumerate} \item[($N=1$)] degree $d$ covers $\sigma\,:\,\Sigma\rightarrow E$ ramified over the origin with cycle type $(1^{d-|\mu|},\mu)$.
\item[($N=2$)] degree $2d$ covers $\sigma\,:\,\Sigma\rightarrow \mathbb{P}^1$ ramified over four points with cycle types $2^d$, $2^d$, $2^d$, and $(2^{d-|\mu|/2},\mu)$.  
\item[($N=3$)] degree $3d$ covers $\sigma\,:\,\Sigma\rightarrow \mathbb{P}^1$ ramified over three points with cycle types $3^d$, $3^d$, and $(3^{d-|\mu|/3},\mu)$. 
\item[($N=4$)] degree $4d$ covers $\sigma\,:\,\Sigma\rightarrow \mathbb{P}^1$ ramified over three points with cycle types $2^{2d}$, $4^d$, and $(4^{d-|\mu|/4},\mu)$. 
\item[($N=6$)] degree $6d$ covers $\sigma\,:\,\Sigma\rightarrow \mathbb{P}^1$ ramified over three points with cycle types $2^{3d}$, $(3^{2d-|\mu|/6},\frac{1}{2}\mu)$, and $6^d$. 
\end{enumerate} In addition, the triangulated surfaces with $2d$ triangles and curvatures $\kappa$ are in bijection with covers of $\mathbb{P}^1$ of degree $6d$ ramified over three points with cycle types $2^{3d}$, $3^{2d}$, and $(6^{d-|\mu|},\mu)$. \end{proposition}

\begin{proof} First we show that a tiling produces a branched cover. Let $\Sigma_0\subset \Sigma$ be complement of the vertices of $\mathcal{T}$. Locally, $\Sigma_0$ admits an oriented isometric embedding into $\R^2\cong \C$ which is unique up to the action of the group of oriented isometries $U(1)\ltimes \C$. Given two such charts $\pi_U\,:\,U\rightarrow \C$ and $\pi_V\,:\,V\rightarrow \C$ on contractible open sets, there is a unique orientation-preserving isometry $\phi\in U(1)\ltimes \C$ of the plane such that $\phi\circ \pi_U=\pi_V$ on $U\cap V$. By post-composing $\pi_U$ with $\phi$ we can ensure that the charts on $U$ and $V$ agree on their overlap. Gluing local charts together, we get a {\it developing map} from the universal cover $$\pi\,:\,\widetilde{\Sigma_0}\rightarrow \C$$ which is unique up to post-composition by an element of $U(1)\ltimes \C$. Pulling back the complex structure from $\C$ induces a complex structure on $\widetilde{\Sigma_0}$ which descends to $\Sigma_0$. This complex structure extends to the vertices.

The holonomy of the resulting metric lies in $G$, so we may assume that $\pi\,:\,\widetilde{\Sigma_0}\rightarrow \C$ maps the inverse image of $\mathcal{T}$ into a fixed tiling of $\C$ by $N$-duck tiles. Then $\pi$ is unique up to post-composition by an element of the crystallographic group $G=\langle \zeta_N\rangle\ltimes \Lambda$, where $\Lambda$ is the translation group preserving the tiling, and $\zeta_N$ acts by multiplication. We conclude that there is a map

$$\Sigma_0\rightarrow \C/G=X_N\cong \left\{
		\begin{array}{lllll}
			E & \mbox{if } N=1 \\
			\mathbb{P}_{2,2,2,2} & \mbox{if } N=2 \\
			\mathbb{P}_{3,3,3} & \mbox{if } N=3 \\
			\mathbb{P}_{2,4,4} & \mbox{if } N=4 \\
			\mathbb{P}_{2,3,6} & \mbox{if } N=6. \\
		\end{array}
	\right.$$

Here $\mathbb{P}_{a,b,c}$ denotes a projective line with orbifold points of the specified orders and $\mathbb{P}_{2,2,2,2}$ is the so-called {\it pillowcase} of the elliptic curve $E$ which double covers it. Forgetting the orbifold structure, we may continuously extend the above map from $\Sigma_0$ to all of $\Sigma$, and get a branched covering of either an elliptic curve (for $N=1$) or $\mathbb{P}^1$ (for $N=2,3,4,6$). This extension is holomorphic by Riemann's theorem on removable singularities.

All vertices of the tiling map to a specified point $p$, and the map from $\Sigma_0\rightarrow X_N-\{p\}$ is orbifold-unramified. The flat metric on $\Sigma$ is the pullback of the flat metric on $X_N$. The ramification orders can thus be read off from the cone angles of the metric on $\Sigma$. Since $\Sigma_0$ has no cone points, the order of ramification of a point in $\Sigma_0$ is constant over a given point $q\in X_N -\{p\}$, and equal to the order of ramification of the orbifold chart centered at $q$. Over $p$, the cone angles of the metric on $\Sigma$ are all integer multiples of the cone angle at $p\in X_N$, and this integer multiple is the ramification order.

Thus, we have seen how to produce a branched cover from a tiled surface. Conversely, given a covering of the elliptic orbifold with the appropriate ramification profile, we may lift the quotient of the tile on each orbifold to produce a tiling of the cover. As above, the cone angles of the pullback flat metric are determined by the orders of ramification.

Finally, we compare the two notions of equivalence: Two branched coverings $\sigma\,:\,\Sigma\rightarrow X$ and $\sigma'\,:\,\Sigma'\rightarrow X$ are equivalent if and only if there is a map $\psi\,:\Sigma\rightarrow \Sigma'$ making the triangle commute. Given two equivalent tilings, the isometry between them produces an isomorphism of branched covers. Conversely, two equivalent branched covers produce the same tiling since the tiling can be reconstructed from the branched cover. In particular, this applies to automorphisms: The group of deck transformations $\textrm{Aut}\,\sigma$ coincides with the oriented automorphism group of the tiling. \end{proof}

\subsection{Review of Hurwitz Theory} We derive the famous result due originally to Frobenius \cite{frobenius} and Schur \cite{schur} enumerating branched covers of $\mathbb{P}^1$ with specified ramification. The corresponding result for elliptic curves is treated in \cite{dijkgraaf,eo1} and relies on a result of Burnside \cite{burnside}.

\begin{definition} The {\it Hurwitz number} \begin{align*} H_d(\eta^1,\dots,\eta^n) &=\sum_{\sigma} \frac{1}{|\textrm{Aut}(\sigma)|}\end{align*} is the weighted count of brached covers $\sigma\colon X\rightarrow \mathbb{P}^1$ of degree $d$ with ramification profiles $\eta^i$ over fixed points $p_i\in \mathbb{P}^1$. \end{definition}

Let $*\in\mathbb{P}^1$ be a base point not equal to any $p_i$. Then $\sigma$ is determined by the monodromy action on the fiber over $*$. Choosing a labelling of the fiber over $*$ by $\{1,\dots,d\}$, we get a monodromy representation $$\rho\,:\, \pi_1(\mathbb{P}^1\backslash\{p_1,\dots,p_n\},*)\rightarrow S_d.$$ Let $\gamma_i$ be simple loops enclosing $p_i$ such that $\gamma_1\dots\gamma_n=1$. The representation $\rho$ is determined by the elements $s_i:=\rho(\gamma_i)\in C_{\eta^i}$ where $C_{\eta^i}$ is the conjugacy class of elements with cycle type $\eta^i$ in $S_d$. Thus, the set of labelled monodromy representations is in bijection with tuples $$\{(s_1,\dots,s_n)\in C_{\eta^1}\times \cdots \times C_{\eta^n}\,:\,s_1\dots s_n=1\}.$$

Let Aut$(\sigma)$ denote the centralizer of the image of $\rho$. It is natural to weight a cover by a factor of $|\textrm{Aut}\,\sigma|^{-1}$ since by orbit-stabilizer, the weighted number of covers is the size of the above set divided by $d!$. 

Let $C_\eta$ denote the sum of the group elements of cycle type $\eta$ in the group algebra $\C[S_d]$. Let $reg$ denote the regular representation of $S_d$. Note that \begin{align*}H_d(\eta^1,\dots,\eta^n) =\frac{1}{(d!)^2}\textrm{tr}_{reg}\prod C_{\eta^i}\end{align*} since the only conjugacy class with nonzero trace in $reg$ is $\{1\}$, and its trace is $d!$. Let $\chi^\lambda$ be the character of the irreducible representation of $S_d$ associated to the partition $\lambda$ and let $\textrm{dim}\,\lambda$ be its dimension. In this irreducible representation, the conjugacy class $C_\eta$ acts by a scalar by Schur's lemma. By taking traces, we see that this scalar is $${\bf f}_\eta(\lambda):=|C_\eta|\,\frac{\chi^\lambda(\eta)}{\textrm{dim}\,\lambda}.$$ The function ${\bf f}_\eta$ is called the {\it central character}. Decomposing the regular representation into irreducibles, we conclude:

\begin{proposition}[Frobenius-Schur]\label{one} The Hurwitz number counting branched covers of $\mathbb{P}^1$ is given by the formula \begin{equation}\label{frobschur} H_d(\eta^1,\dots,\eta^n)=\sum_{|\lambda|=d}\left(\frac{\dim\lambda}{d!}\right)^2\prod {\bf f}_{\eta^i}(\lambda).\end{equation} \end{proposition}

\begin{remark} We have failed to impose the condition that the branched cover is connected. This would require assuming that the group $\langle s_i\rangle\subset S_d$ acts transitively, which is somewhat unnatural representation-theoretically. We will later discuss how to impose connectedness. \end{remark}

Our main objects of study are the following generating functions of Hurwitz numbers:

\begin{definition}\label{general} Let $N=3$, $4$, or $6$. Let $\nu^i$ for $i=1,\dots,n$ be a finite collection of partitions and let $\mu^a, \mu^b, \mu^c$ be three partitions indexed by the orders $a,b,c$ of the orbifold points on $X_N$ with $a\leq b\leq c$. Define $H_N(\nu^i\,|\,\mu^a,\mu^b,\mu^c)$ to be the generating function \begin{align}\label{main}
\sum_\lambda q^{|\lambda|/N}\left(\frac{\dim\lambda}{|\lambda|!}\right)^2{\bf f}_{a,\dots,a,\mu^a}(\lambda) \,{\bf f}_{b,\dots,b,\mu^b}(\lambda) \,{\bf f}_{c, \dots,c,\mu^c}(\lambda)\prod_i {\bf f}_{\nu^i}(\lambda).  \end{align} \end{definition} For $N=1,2$, define analogously \begin{align*} H_1(\nu^i):=&\sum_\lambda q^{|\lambda|} \prod_i {\bf f}_{\nu^i}(\lambda)\,\,\,\,\,\,\,\,\,\,\textrm{ and} \\ H_2(\nu^i\,|\,\mu^a,\mu^b,\mu^c,\mu^d):=&\sum_\lambda q^{|\lambda|/2} \!\!\!\!\prod_{r\in\{a,b,c,d\}} \!\!\!\!{\bf f}_{2,\dots,2,\mu^r}(\lambda)\prod_i{\bf f}_{\nu^i}(\lambda). \end{align*}

We use the notation $D=\{\nu^i\,|\,\mu^a,\mu^b,\mu^c\}$ (or if $N=2$ including a $\mu^d$ also), to denote the {\it ramification profile} of a map to $X_N$. Then, $H_N(D)$ is the generating function for possibly disconnected Hurwitz covers with profile $D$. Let $h^{dis}_N(\kappa,q)$ denote the generating function of possibly disconnected $N$-duck tilings with nonzero curvatures $\kappa$. Propositions \ref{zero} and \ref{one} imply that \begin{align*}h_1^{dis}(\kappa,q)&= H_1(\mu) \\ h_2^{dis}(\kappa,q)&= H_2(\emptyset \,|\,\mu,\emptyset,\emptyset,\emptyset) \\ h_3^{dis}(\kappa,q)&= H_3(\emptyset \,|\,\mu,\emptyset,\emptyset) \\ h_4^{dis}(\kappa,q)&= H_4(\emptyset \,|\,\emptyset,\mu,\emptyset) \\  h_6^{dis}(\kappa,q)&= H_6(\emptyset \,|\,\emptyset,\mu,\emptyset). \end{align*} Similarly, the generating function for possibly disconnected triangulations with curvatures $\kappa$ is $$\triangle^{dis}(\kappa,q)=H_6(\emptyset \,|\,\emptyset,\emptyset,\mu).$$

\section{Quotients, cores, and the half-infinite wedge}\label{quotients}

\subsection{Quotients and Cores} Throughout this section, we identify a partition $\lambda=\{\lambda_1\geq \lambda_2\geq \dots\}$ with its Young diagram---a collection of unit boxes in the fourth quadrant of the plane whose $i$th row has length $\lambda_i$. 

\begin{definition} A {\it $t$-rim hook} $\nu$ of $\lambda$ is a contiguous string of $t$ boxes along the jagged edge of the Young diagram of $\lambda$ whose complement $\lambda\backslash \nu$ is still a Young diagram. The {\it height} $ht(\nu)$ is the physical height of $\nu$ in the plane. \end{definition}

See Figure \ref{hookpic} for an example of a $4$-rim hook of $\lambda=(6,5,2,2,1)$ of height $2$. The Murnagan-Nakayama rule \cite{murnaghan,nakayama} gives a beautiful inductive algorithm for computing the characters $\chi^\lambda(\mu)$ of $S_d$. It states: $$\chi^\lambda(\mu_1,\dots,\mu_\ell)=\sum_{\mu_1\textrm{-rim hooks }\nu}(-1)^{ht(\nu)+1}\chi^{\lambda\backslash \nu}(\mu_2,\dots,\mu_\ell).$$ Equivalently, $\chi^\lambda(\mu)$ is the signed number of {\it tableaux of shape $\lambda$ and content $\mu$}---that is, the signed number of ways to decompose $\lambda$ by first removing a rim hook of size $\mu_1$, then one of size $\mu_2$, etc. Here the sign of the tableau is the product of the signs $(-1)^{ht(\nu)+1}$ for each rim hook in the decomposition. Perhaps surprisingly, the ordering of the $\mu_i$ is irrelevant to the final answer. A special case states that $\dim \lambda = \chi^\lambda(1,\dots,1)$ is the number of standard Young tableaux of shape $\lambda$.

We will be particularly interested in characters of the form $\chi^\lambda(t,\dots,t)$. In this case, the sign of the tableau is the same for all so-called $t$-rim hook decompositions, so that $$|\chi^\lambda(t,\dots,t)|=\#\{t\textrm{-rim hook tableaux of shape }\lambda\}$$ is non-vanishing if and only if $\lambda$ admits a decomposition into $t$-rim hooks.

\begin{definition} The partition $\lambda$ is {\it $t$-decomposable} if its shape admits a decomposition into $t$-rim hooks. Let $\textrm{sgn}_t(\lambda)$ denote the sign of $\chi^\lambda(t,\dots,t)$ for a $t$-decomposable partition $\lambda$. \end{definition} 

Even when $\lambda$ is not $t$-decomposable, one may remove $t$-rim hooks from $\lambda$ until no longer possible. Then, the resulting shape is unique regardless of the manner in which the $t$-rim hooks were removed:

\begin{definition} The {\it $t$-core} $\lambda^{\textrm{mod }t}$ is the result of removing as many $t$-rim hooks as possible from $\lambda$. \end{definition}

\begin{figure}
\includegraphics[width=1.7in]{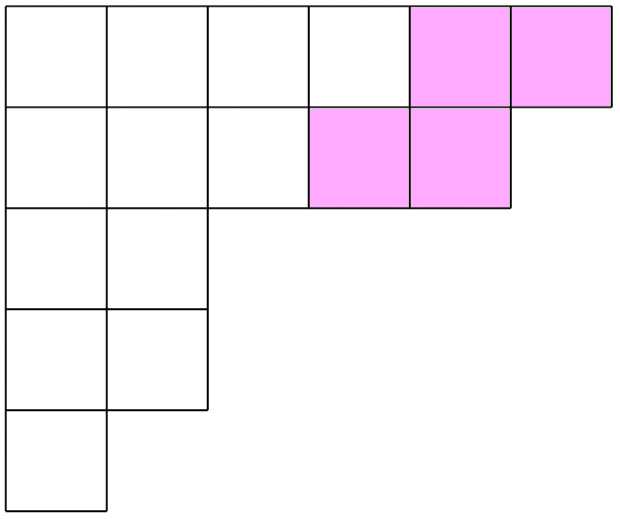}
\caption{A $4$-rim hook of the partition $\lambda=(6,5,2,2,1)$.}
\label{hookpic}
\end{figure}

\begin{figure}
\includegraphics[width=1.7in]{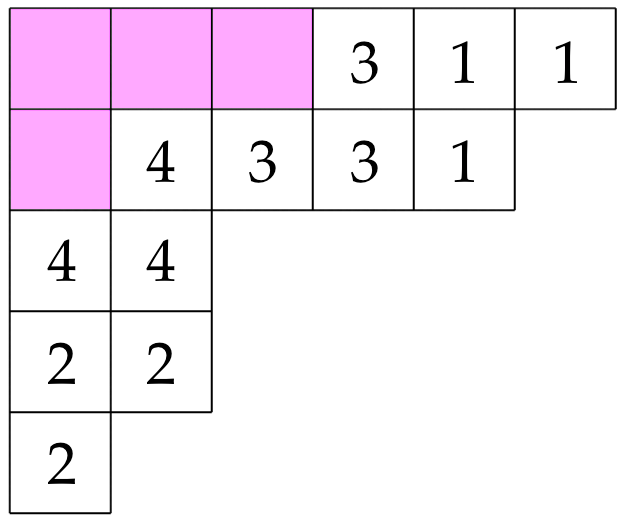}
\caption{Successively removable $3$-rim hooks from the Young diagram of the partition $\lambda=(6,5,2,2,1)$. The $3$-core $\lambda^{\textrm{mod }3}=(3,1)$ has been shaded.}
\label{fig3}
\end{figure}

\begin{figure}
\includegraphics[width=2.7in]{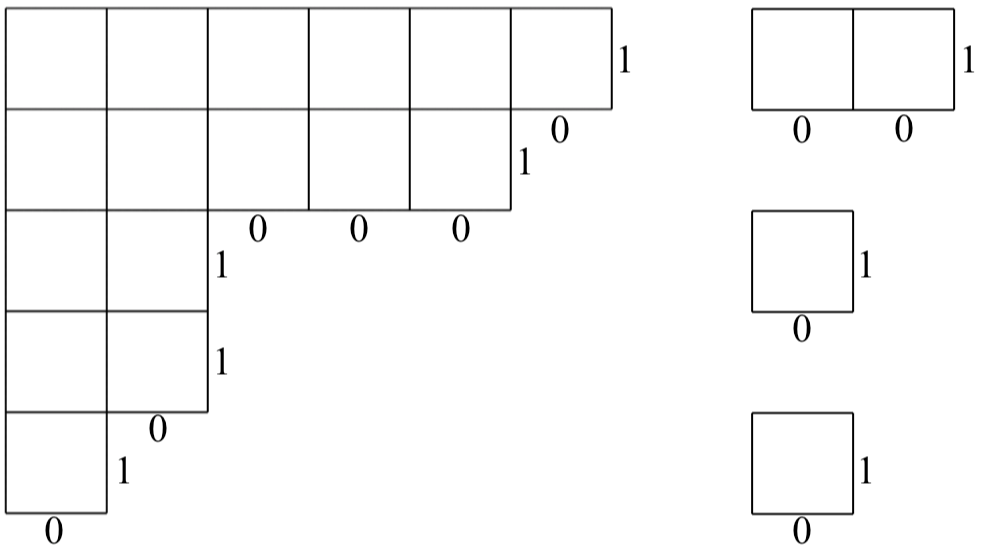}
\caption{On the left is the partition $\lambda=(6,5,2,2,1)$ and its (01)-sequence. On the right are the $3$-quotients $\lambda^{0/3}$, $\lambda^{1/3}$, and $\lambda^{2/3}$ and their $(01)$-sequences.}
\label{fig4}
\end{figure}

The removal of a $t$-rim hook of $\lambda$ is best understood in terms of the {\it $(01)$-sequence} associated to $\lambda$. This is the bi-infinite sequence of zeroes and ones which determine successively whether one goes left or down, respectively, along the boundary of the Young diagram as one proceeds from top right to bottom left. For instance, as shown in Figure \ref{fig4}, the $(01)$-sequence for $\lambda = (6,5,2,2,1)$ is $$\dots0\,\,0\,\,0\,\,0\,\,0\,\,0\,\,1\,\,0\,\,1\,\,0\,\,0 \,\,0\,\big{|}\,1\,\,1\,\,0\,\,1\,\,0\,\,1\,\,1\,\,1\,\,1\,\,1\,\,1\,\,1\,\,\dots$$ where the bar $\big{|}$ denotes the diagonal line emanating from the upper left corner of the Young diagram. The sequence begins with an infinite string of $0$'s, and terminates in an infinite string of $1$'s.

Note the number of $1$'s before the bar is equal to the number of $0$'s after the bar; we say the sequence has {\it charge zero}. We think of the terms in the $(01)$-sequence as slots, with $1$ or $0$ indicating whether the slot is occupied or unoccupied by a bead, respectively. Place a bar at zero on the real number line, so that the slots lie at the set of half-integers $\frac{1}{2}+\Z$, but with the positive real axis going left (this convention is ubiquitous throughout the subject so we maintain it). Then, the occupied slots are those indexed by the half integers $\lambda_i-i+\frac{1}{2}$.

The $t$-rim hooks of $\lambda$ then correspond to a $1$ whose position is $t$ larger than a $0$. So the $t$-rim hooks are in bijection with ways to hop a bead to the right by $t$ into an empty slot, and the sign $(-1)^{ht(\nu)+1}$ of the associated $t$-rim hook is the number of beads over which this bead hops. Furthermore the uniqueness of the $t$-core is apparent: Split the $(01)$-sequence into $t$ periodic subsequences corresponding to the congruence classes of the slot positions mod $t$. For example, when $t=3$ and the partition is $\lambda=(6,5,2,2,1)$ the subsequences are: \begin{align*}
&\dots\,\,0\,\,\,\,\,\,0\,\,\,\,\,\,1\,\,\,\,\,\,0\,\,\,\,\,\big{|}\,1\,\,\,\,\,\,1\,\,\,\,\,\,1\,\,\,\,\,\,1\,\,\dots\,\,\,\,\hspace{10pt} \lambda_i-i+\tfrac{1}{2}\equiv \tfrac{1}{2}\,\,\rm{mod}\,\,3  \\
&\,\,\dots\,\,0\,\,\,\,\,\,0\,\,\,\,\,\,0\,\,\,\,\,\,0\,\,\,\big{|}\,\,\,1\,\,\,\,\,\,0\,\,\,\,\,\,1\,\,\,\,\,\,1\,\,\dots\,\,\hspace{10pt} \lambda_i-i+\tfrac{1}{2}\equiv \tfrac{3}{2}\,\,\rm{mod}\,\,3\\
&\,\,\,\,\dots\,\,0\,\,\,\,\,\,0\,\,\,\,\,\,1\,\,\,\,\,\,0\,\big{|}\,\,\,\,\,0\,\,\,\,\,\,1\,\,\,\,\,\,1\,\,\,\,\,\,1\,\,\dots \hspace{10pt} \lambda_i-i+\tfrac{1}{2}\equiv \tfrac{5}{2}\,\,\rm{mod}\,\,3.
 \end{align*} 

To move a bead $t$ to the right is the same as choosing one of the subsequences, and moving a bead $1$ to the right. Thus, removing as many $t$-rim hooks as possible corresponds to moving all beads on each substring as far to the right as possible. This procedure is called {\it clearing the $t$-abacus}, and the resulting $(01)$-sequence is that of the $t$-core $\lambda^{\textrm{mod }t}$. Furthermore, we may also associate to $\lambda$ a collection of $t$ partitions by looking at each substring:

\begin{definition} The {\it $t$-quotients} $\lambda^{a/t}$ for $0\leq a<t$ are the partitions whose $(01)$-sequences are indexed by the slot positions congruent to $a+\tfrac{1}{2} \textrm{ mod }t$. \end{definition}

The $t$-cores are in bijection with the $A_{t-1}$ lattice $\left\{\sum x_a=0\right\}\subset \Z^t$ as follows: Each $(01)$-subsequence mod $t$ of a $t$-core is a string of zeroes followed by a string of ones. Let $x_a$ be the difference between the junction between zeroes and ones on the $a$th subsequence the location of the original bar $\big{|}$ of $\lambda$. Then the point $(x_0,\dots,x_{t-1})\in \Z^t$ satisfies $\sum x_i=0$ because the original partition $\lambda$ has charge zero. For instance, in the above example, $(x_0,x_1,x_2)=(1,-1,0).$ 

The {\it hook} of a box in $\lambda$ is the union of the box with those boxes either directly below or to the right of it. The {\it hook length} is the number of boxes forming the hook. The hooks of length $t$ are in bijection with rim-hooks of length $t$: Take the rim-hook which shares the same ends as the hook.

\begin{proposition}\label{quotient} The hook lengths of $\lambda$ which are divisible by $t$ are exactly the hook lengths of the $\lambda^{a/t}$ multiplied by $t$. Suppose furthermore that $\lambda$ is $t$-decomposable. Then \begin{align*}|\chi^\lambda(t,\dots,t)|&=\#\left\{\textrm{standard tableaux with shape }\coprod \lambda^{a/t}\right\} \\ &={|\lambda|/t \choose |\lambda^{0/t}|\,\cdots\,|\lambda^{t-1/t}|}\prod_{i=0}^{t-1}\dim\lambda^{a/t} \\ &=\frac{t^{|\lambda|/t}(|\lambda|/t)!}{\prod \textrm{hook lengths of }\lambda\textrm{ divisible by }t}.\end{align*} \end{proposition} 
\begin{proof} It follows directly from the discussion of the $(01)$-sequence that the hooks (or equivalently the rim-hooks, or the boxes) of $\lambda$ are in bijection with pairs of an occupied slot appearing to the left of an unoccupied slot, and the distance between these slots is the hook length. The first two formulae follow directly. See also \cite{fomin}. The third equality follows from the hook length formula, which states that $\frac{|\lambda|!}{\dim\lambda}$ is the product of the hook lengths of $\lambda$. \end{proof}

\subsection{Fock space formalism} The Fock space formalism encodes the above observations about $(01)$-sequences. See Okounkov-Pandharipande \cite{op} or Rios-Zertuche \cite{rz2} for further exposition. Let $$V:=\textrm{span}\,\{\underline{i}\,:\,i\in\tfrac{1}{2}+\Z\}$$ be an infinite-dimensional vector space with a basis of symbols $\underline{i}$ indexed by all half-integers.

\begin{definition} The {\it Fock space} or {\it half-infinite wedge} $\Lambda^{\infty/2}V$ is the space spanned by formal symbols $$\underline{i_1}\wedge \underline{i_2}\wedge \underline{i_3}\wedge \dots$$ over all strictly decreasing sequences $i_1>i_2>i_3>\dots$ of half integers such that $i_{j+1}=i_j-1$ for all $j\gg 0$. \end{definition} 

Define an inner product $\langle \cdot | \cdot \rangle$ on Fock space by declaring these symbols to be orthonormal. The {\it charge $C$ subspace} $\Lambda_C^{\infty/2}V$ is the sub-span of basis vectors for which $i_j=-j+\frac{1}{2}+C$ for all $j\gg 0$. Then, the charge zero subspace $\mathbb{L}:=\Lambda^{\infty/2}_0V$ has a basis naturally indexed by partitions of all integers (including the empty partition of zero): $$v_\lambda:=\underline{\lambda_1-\tfrac{1}{2}}\wedge \underline{\lambda_2-\tfrac{3}{2}}\wedge \underline{\lambda_3-\tfrac{5}{2}}\wedge \dots.$$ Denote by $v_\emptyset$ the {\it vacuum vector}, for which $\lambda_i=0$ for all $i$. The values of $i_j$ for which $\underline{i_j}$ appears in $v_\lambda$ are exactly the occupied slots discussed in the previous section.

\begin{definition} The {\it Heisenberg algebra} is the Lie algebra $$\C e\oplus \bigoplus_{n\neq 0}\C\alpha_n$$ such that $e$ is central and $[\alpha_n,\alpha_m]=n\delta_{n,-m}e$. \end{definition}

The Heisenberg algebra acts on Fock space and restricts to an action on $\mathbb{L}$ by the action on basis vectors \begin{align*}\alpha_n\,:\,\underline{i_1}\wedge \underline{i_2}\wedge \underline{i_3}\wedge\dots\mapsto &\, \underline{i_1-n}\wedge \underline{i_2}\wedge \underline{i_3}\wedge\dots+ \\ & \,\underline{i_1}\wedge \underline{i_2-n}\wedge \underline{i_3}\wedge \dots+ \dots \end{align*} where the usual rules of a wedge product are used to rewrite the right-hand side in terms of the basis of Fock space. Note that since $i_{j+1}=i_j-1$ for all $j\gg 0$, the righthand sum is in fact finite. The element $e$ acts by multiplication by $1$, so we say the representation has {\it central charge} $1$.

The sign of a rim hook exactly corresponds to how the sign of a wedge changes by reordering terms, so we have the following succinct rephrasing of the Murnaghan-Nakayama rule: $$\chi^\lambda(\mu)=[v_\emptyset] \prod_i \alpha_{\mu_i} v_\lambda = [v_\lambda] \prod_i \alpha_{-\mu_i}v_\emptyset.$$ Observe that $\alpha_n$ and $\alpha_{-n}$ are adjoint. We call $v_\lambda$ the {\it fermionic basis} of $\mathbb{L}$. Define the {\it bosonic basis}, also indexed by partitions, to be $$ |\mu\rangle:=\prod \alpha_{-\mu_i}v_\emptyset.$$ Let $H$ be the {\it energy operator} on $\mathbb{L}$, which acts by $Hv_\lambda = |\lambda|v_\lambda$. Then $v_\lambda$ and $| \mu\rangle$ are both eigenbases for the action of $H$, and the change-of-basis between the fermionic and bosonic bases at energy level $d$ (that is, on the eigenspace of $H$ with eigenvalue $d$) is the matrix of inner products $$\langle v_\lambda|\mu\rangle=\chi^\lambda(\mu)$$ i.e. exactly the character table of $S_d$. 

The bijective correspondence between a partition and its collection of $t$-quotients and $t$-core can be rephrased in terms of an isomorphism $$\mathbb{L}\xrightarrow{\sim} \mathbb{L}^{\otimes\hspace{0.2pt} t} \otimes \C[A_{t-1}]$$ where $\C[A_{t-1}]$ is the group algebra. This isomorphism sends $$v_\lambda\mapsto\pm (v_{\lambda^{0/t}}\otimes  \cdots \otimes  v_{\lambda^{t-1/t}})\otimes \lambda^{\textrm{mod }t}$$ where the $t$-core $\lambda^{\textrm{mod }t}$ is identified with the associated lattice point in $A_{t-1}$. 

\section{Shifted-symmetric polynomials}\label{shiftedsymm}

\subsection{Bases of shifted-symmetric polynomials} As in \cite{eo1}, our starting point in the analysis of (\ref{main}) is a theorem of Kerov and Olshanski that ${\bf f}_\nu(\lambda)$ is a shifted-symmetric polynomial.

\begin{definition} Let $\Lambda^*(n)$ be the ring of {\it shifted-symmetric polynomials} in $n$ variables, i.e. polynomials in the variables $\lambda_1,\dots,\lambda_n$ which are symmetric in $\lambda_i-i$. Then $\Lambda^*(n+1)$ maps to $\Lambda^*(n)$ by setting $\lambda_{n+1}=0$. Define $\Lambda^*:=\displaystyle\lim_{\longleftarrow} \Lambda^*(n)$ to be the inverse limit of these algebras. \end{definition}

As in the usual ring of symmetric polynomials, the algebra $\Lambda^*$ has a number of natural bases, such as the monomials in the power sums $${\bf p}_k(\lambda)=(1-2^{-k})\zeta(-k)+\sum_i \left[ \left(\lambda_i-i+\tfrac{1}{2}\right)^k-\left(-i+\tfrac{1}{2}\right)^k\right].$$ The reason for the unusual constant is that one would like to define $${\bf p}_k(\lambda)``=" \sum_i (\lambda_i-i+\tfrac{1}{2})^k$$ but naively, this sum is divergent when evaluated on any partition, since $\lambda_i=0$ for all $i\gg 0$. Hence one subtracts an appropriate constant inside each summand, and adds back the renormalized infinite sum on the outside. The monomials ${\bf p}_\mu:=\prod {\bf p}_{\mu_i}$ form a basis of $\Lambda^*$. We have

\begin{theorem}[Kerov-Olshanski, \cite{kerov}]\label{kero} For any partition $\nu$, the function ${\bf f}_{\nu}(\lambda)$ on partitions lies in the ring of shifted-symmetric polynomials ${\bf f}_\nu\in \Lambda^*$. \end{theorem}

Here if $|\nu|< |\lambda|$, we set $${\bf f}_\nu(\lambda):={|\lambda| \choose |\nu|}|C_\nu| \frac{\chi^\lambda(\nu,1,\dots,1)}{\dim \lambda},$$ and if $|\nu|>|\lambda|$ then ${\bf f}_\nu(\lambda)=0$. Theorem \ref{kero} implies that for any $\nu$, the function ${\bf f}_\nu$ is uniquely expressible as a polynomial in the ${\bf p}_k$'s. It is best to package the functions ${\bf p}_k(\lambda)$ together as the Taylor coefficients of a function $${\bf e}(\lambda,z):=\sum_i e^{(\lambda_i-i+\frac{1}{2})z}=\frac{1}{z}+\sum {\bf p}_k(\lambda)\frac{z^k}{k!}.$$ Then, another formulation of the theorem of Kerov and Olshanski is that ${\bf f}_\nu(\lambda)$ is a finite linear combination of Taylor coefficients $$[z_1^{k_1}\dots z_n^{k_n}]\, {\bf e}(\lambda,z_1)\dots {\bf e}(\lambda,z_n).$$

The ring $\Lambda^*$ also has shifted analogues of the Schur functions, introduced by Okounkov and Olshanski \cite{oo}. First define a {\it skew diagram} to be the complement of one Young diagram which is contained in another. If $\eta\subset \lambda$, denote the skew diagram by $\lambda\backslash \eta$. We may extend the Murnaghan-Nakayama rule, defining $\chi^{\lambda\backslash \eta} (\nu)$ to be the signed number of tableaux of shape $\lambda\backslash \eta$ with content $\nu$. For instance, we define $\dim \lambda\backslash \eta:=\chi^{\lambda\backslash \eta}(1,\dots,1)$. 

\begin{definition} The {\it shifted Schur function} ${\bf s}_\eta(\lambda)$ are defined by the property ${\bf s}_\eta(\lambda)=0$ unless $\eta\subset \lambda$, in which case \begin{align}\label{six}\frac{\dim \lambda\backslash \eta}{|\lambda\backslash \eta|!}=\frac{\dim \lambda}{|\lambda|!}{\bf s}_\eta(\lambda).\end{align}  \end{definition}

Both the ${\bf f}_\nu$ and ${\bf s}_\eta$ form bases of $\Lambda^*$. 

\subsection{Characters and quotients} We now analyze equation \ref{main}, in particular, the functions ${\bf f}_{t,\dots,t,\mu}(\lambda)$ which encode the branching over the orbifold point on $X_N$ of order $t$.

\begin{remark} Let $\lambda\backslash \eta$ be a $t$-decomposable skew shape. Then the quotient shapes $(\lambda\backslash \eta)^{a/t}:=\lambda^{a/t}\backslash \eta^{a/t}$ are well-defined, and the analogue of Proposition \ref{quotient} holds for the skew character $\chi^{\lambda\backslash \eta}(t,\dots,t)$. \end{remark}

Define $\mathfrak{z}(\mu):=|{\rm Aut}\,\mu|\prod \mu_i$. We have \begin{align*} {\bf f}_{t,\dots,t,\mu}(\lambda) & =|C_{t,\dots,t,\mu}| \frac{\chi^\lambda(t,\dots,t,\mu)}{\dim \lambda} \\ &=\frac{|\lambda|!}{t^{|\lambda\backslash \mu|/t}(|\lambda\backslash \mu|/t)!\mathfrak{z}(\mu)\dim \lambda}\!\!\!\!\! \sum_{\substack{ |\eta|=|\mu| \\ \eta^{{\rm mod}\,t}=\lambda^{{\rm mod}\,t}}} \!\!\!\!\!\chi^\eta(\mu) \chi^{\lambda\backslash \eta}(t,\dots,t) \end{align*} where the second equality follows from the Murnaghan-Nakayama rule---we may remove the $t$-ribbons first, and then decompose remaining shape $\eta$ into $\mu_i$-ribbons. Now we apply the analogue of Proposition \ref{quotient} to conclude that $${\bf f}_{t,\dots,t,\mu}(\lambda)=\frac{|\lambda|!}{t^{|\lambda\backslash \mu|/t}\mathfrak{z}(\mu)\dim \lambda}\!\!\!\!\! \sum_{\substack{ |\eta|=|\mu| \\ \eta^{{\rm mod}\,t}=\lambda^{{\rm mod}\,t}}} \!\!\!\!\!\chi^\eta(\mu)\frac{{\rm sgn}_t(\lambda)}{{\rm sgn}_t(\eta)} \prod_{a\textrm{ mod }t} \frac{ \dim (\lambda\backslash \eta)^{a/t}}{|(\lambda\backslash \eta)^{a/t}|!}$$ where the definition of ${\rm sgn}_t(\lambda)$ is extended to shapes with non-empty $t$-core by declaring it to be the sign of a $t$-ribbon tableau of shape $\lambda\backslash \lambda^{{\rm mod}\,t}$.

\begin{definition} Define the function $${\bf c}_t(\lambda)=\frac{\textrm{sgn}_t(\lambda)}{t^{|\lambda|/t}}\cdot \frac{|\lambda|!}{\dim \lambda}\prod_{a\textrm{ mod }t} \frac{\dim \lambda^{a/t}}{|\lambda^{a/t}|!}.$$ \end{definition}

The utility of this function is that it is visibly independent of $\mu$ but is in some sense ``close" to ${\bf f}_{t,\dots,t,\mu}(\lambda)$. In particular, \ref{six} implies the following expression for their ratio:

\begin{equation}\label{gdef} \frac{{\bf f}_{t,\dots,t,\mu}(\lambda)}{{\bf c}_t(\lambda)}=\frac{t^{|\mu|/t}}{\mathfrak{z}(\mu)}\!\!\!\!\! \sum_{\substack{ |\eta|=|\mu| \\ \eta^{{\rm mod}\,t}=\lambda^{{\rm mod}\,t}}} \!\!\!\!\!\chi^\eta(\mu){\rm sgn}_t(\eta)\prod_{a\textrm{ mod }t}{\bf s}_{\eta^{a/t}}(\lambda^{a/t})\end{equation}

\begin{remark} Rios-Zertuche \cite{rz} observed that this equality for $t=2$ simplifies part of the proof of Theorem 2 of \cite{eo2}. \end{remark}  

We are led to the following definition, which generalizes the $t=2$ and $\lambda^{{\rm mod}\,2}=\emptyset$ case in \cite{eo2}:

\begin{definition} Define ${\bf g}^t_\mu(\lambda):={\bf c}_t(\lambda)^{-1}{\bf f}_{t,\dots,t,\mu}(\lambda)$. \end{definition}

\begin{remark}\label{irrational} Note that ${\bf g}_\mu^t(\lambda)$ may be irrational, when $t$ does not divide $|\mu|$. But then ${\bf c}_t(\lambda)$ will also be irrational. Taking the floor of the power $|\mu|/t$ resolves this issue, but makes the formulas slightly less clean. \end{remark}

We make the following definition, which combines the factors in (\ref{main}) not involving the ${\bf g}$ functions.

\begin{definition}\label{w} Let $\eta$ be an $N$-core and suppose $N=2,3,4,6$. Define the {\it unnormalized $\eta$-weight} of $\lambda$ to be $${\bf \tilde{w}}_{N,\eta}(\lambda):= \twopartdefotherwise{\displaystyle\left(\frac{\dim \lambda}{|\lambda|!}\right)^2\prod_t{\bf c}_t(\lambda)}{\lambda^{\textrm{mod }N}=\eta}{0}$$ The product ranges over the orders $t$ of the orbifold points on $X_N$. \end{definition}

Returning to the analysis of equation (\ref{main}), we may substitute ${\bf f}_{t,\dots,t,\mu}(\lambda)$ for ${\bf g}_\mu^t(\lambda)\cdot {\bf c}_t(\lambda)$, and collect terms for which the core $\eta=\lambda^{{\rm mod}\,N}$ is fixed. This gives a formula for the Hurwitz numbers of $X_N$:

\begin{proposition}\label{modify} Let $N=3,4,6$. Fix an $N$-core $\eta$ and ramification data $D=\{\nu^i\,|\,\mu^a, \mu^b, \mu^c\}$ as in Definition \ref{general}. Let $${\bf F}^D_\eta(\lambda):=\prod {\bf f}_{\nu^i}(\lambda)\!\!\!\!\!\!\!\!\sum_{\substack{|\eta^t|=|\mu^t| \\ (\eta^t)^{{\rm mod }\,t}=\eta^{{\rm mod }\,t}}}\!\! \prod_{t\,\in\,\{a,b,c\}}{\bf g}^t_{\eta^t}(\lambda).$$ Note that ${\bf F}^D_\eta(\lambda)$ is non-zero for only finitely many $\eta$ since the sum is vacuous unless $|\eta|\leq |\mu^c|$. Then, we have \begin{align*} H_N(D)=\sum_{N\textrm{-cores }\eta}\sum_{\lambda} q^{|\lambda|/N}{\bf \tilde{w}}_{N,\eta}(\lambda){\bf F}^D_\eta(\lambda).\end{align*}  An analogous statement holds when $N=2$ for $D=\{\nu^i\,|\,\mu^a,\mu^b,\mu^c,\mu^d\}$. \end{proposition}

\subsection{Enlargements of $\Lambda^*$} We now analyze the first ingredient in the formula of Proposition \ref{modify}, namely the functions ${\bf F}^D_\eta(\lambda)$. By (\ref{gdef}), the ${\bf g}^t_\eta(\lambda)$ is expressible in terms of shifted-symmetric polynomials in the $t$-quotients of $\lambda$ for $t\big{|}N$. But functions like ${\bf s}_{\eta^{a/t}}(\lambda^{a/t})$ do not lie in $\Lambda^*$. Thus, we must enlarge $\Lambda^*$, as in Eskin-Okounkov \cite{eo2}, to include ${\bf F}_\eta^D(\lambda)$. Hence we define \begin{align*} {\bf p}_k^r(\lambda):&=\zeta_N^{r/2} k![z^k]\,{\bf e}(\lambda,z+\tfrac{2\pi r i}{N}) \\
&= c_k^r+\sum_{i>0} \left[ \zeta_N^{r(\lambda_i-i+1)}\left(\lambda_i-i+\tfrac{1}{2}\right)^k-\zeta_N^{r(1-i)}\left(-i+\tfrac{1}{2}\right)^k\right] \end{align*}
where the constant $c_k^r$ regularizes the constant part of the infinite sum in terms of values of Dirichlet $L$-functions evaluated at $-k$. Alternatively, we may define $c_k^r$ by the series expansion $$\frac{1}{e^{z/2}-\zeta_N^{-r}e^{-z/2}}=\sum_k c_k^r\frac{z^k}{k!}$$ gotten by plugging in $\lambda=\emptyset$. We have suppressed the dependence on the integer $N$ to avoid excessively indexed notation. Note that ${\bf p}_k^0={\bf p}_k$ and in the case $N=2$, the function $\overline{\bf p}_k$ of \cite{eo2} equals our ${\bf p}_k^1$.

\begin{definition} The {\it $N$-cyclotomic extension} of the ring of shifted symmetric polynomials is $$\Lambda_N^*:=\Q(\zeta_N)[{\bf p}_k^r]_{k\geq 1,\,r\in\{0,\dots,N-1\}}.$$ Let $\textrm{wt}\,{\bf p}_k^r=k+\delta_r$ define the weight grading on $\Lambda_N^*$. \end{definition}

Then we have the following useful proposition:

\begin{proposition}\label{iscyclo} Fix an $N$-core $\eta$ and suppose $\lambda^{\textrm{mod }N}=\eta$. Let $t\big{|}N$. The shifted symmetric polynomials ${\bf p}_k(\lambda^{a/t})$ in the $t$-quotients of $\lambda$ are represented by an element of $\Lambda_N^*$. \end{proposition} 

\begin{proof} It suffices to prove the claim for $t=N$ since $\Lambda_t\subset \Lambda_N$. Using the fact that $1+\zeta_N^r+\dots+\zeta_N^{(N-1)r}=0$ if and only if $N\nmid r$, we can find a linear combination $$\textrm{const}+\sum b_{k,r} {\bf p}_k^r(\lambda)=\sum_{\substack{{\lambda_i-i\,\equiv\, a} \\ \textrm{mod }N}}(\lambda_i-i-a+\tfrac{N}{2})^k-(-i-a+\tfrac{N}{2})^k$$ for any given representative of $a$ mod $N$. Choosing an appropriate $a$, dividing by $N^k$, and adding a constant gives ${\bf p}_k(\lambda^{a/N})$. Here we use that the $N$-core of $\lambda$ is fixed---$\eta$ uniquely determines which representative of $a\textrm{ mod }N$ we must choose to get ${\bf p}_k(\lambda^{a/N})$. \end{proof}

By Proposition \ref{iscyclo}, the functions ${\bf F}^D_\eta(\lambda)$ appearing in Proposition \ref{modify} are represented by an element of $\Lambda_N^*$ (after scaling by appropriate fractional powers of $t$, see Remark \ref{irrational}).

\subsection{The $\eta$-weight} Next, we analyze the unnormalized $\eta$-weight appearing in Proposition \ref{modify}. There is an elegant formula in terms of hook lengths:

\begin{proposition}\label{hooks} Let $N=2$, $3$, $4$, or $6$. Let $\langle a\rangle$ denote the product of the hook lengths of $\lambda$ congruent to $a$ mod $N$. Assume that $\lambda^{{\rm mod }\,N}=\eta$. Then \begin{align*} {\bf \tilde{w}}_{N,\eta}(\lambda) =\frac{\langle 1\rangle \langle N-1\rangle }{\langle 0\rangle^2}\cdot \prod_t {\rm sgn}_t(\lambda)t^{-|\lambda^{{\rm mod }\,t}|/t} .\end{align*}
\end{proposition}

\begin{proof} We prove the formula in the case $N=6$, with the other cases being similar. By definition, we have \begin{align*} {\bf \tilde{w}}_{N,\eta}(\lambda)&=\left(\frac{\dim \lambda}{|\lambda|!}\right)^2\prod_{t=2,3,6}{\bf c}_t(\lambda) \\ &=\frac{|\lambda|!}{\dim \lambda}\prod_{t=2,3,6}\frac{{\rm sgn}_t(\lambda)}{t^{|\lambda|/t}}\prod_{a\,{\rm mod}\,t}\frac{\dim \lambda^{a/t}}{|\lambda^{a/t}|!} \end{align*} The hook length formula states $$\frac{|\lambda|!}{\dim \lambda}=\prod\, \textrm{hook lengths of}\,\lambda.$$ The hook lengths of $\lambda$ which are divisible by $t$ are exactly $t$ times the hook lengths of $t$-quotients $\{\lambda^{a/t}\}$ and there are exactly $(|\lambda|-|\lambda^{{\rm mod}\,t}|)/t$ such hooks. So all hook lengths appear in the numerator of ${\bf \tilde{w}}_{N,\eta}(\lambda)$, whereas the hook lengths visible by $2$, $3$, $6$ each appear in the denominator. Extracting the remaining factors, we conclude
\begin{align*} {\bf \tilde{w}}_{N,\eta}(\lambda)&=\frac{\langle 0\rangle\langle 1\rangle\langle 2\rangle\langle 3\rangle\langle 4\rangle\langle 5\rangle}{\langle 0\rangle\langle 2\rangle\langle 4\rangle\cdot \langle 0\rangle\langle 3\rangle\cdot \langle 0\rangle} \prod_{t=2,3,6} {\rm sgn}_t(\lambda)t^{-|\lambda^{{\rm mod }\,t}|/t} \\ &=\frac{\langle 1\rangle\langle 5\rangle}{\langle 0\rangle^2} \prod_{t=2,3,6} {\rm sgn}_t(\lambda)t^{-|\lambda^{{\rm mod }\,t}|/t}.\end{align*} The fact that one always gets the ratio $\frac{\langle 1\rangle \langle N-1\rangle }{\langle 0\rangle^2}$ is related to the fact that $[\Q(\zeta_N):\Q]\leq 2$ \end{proof} 

Though there are no elliptic orbifolds of orders $N=5$ or $N\geq 7$, Proposition \ref{hooks} invites the following generalization:

\begin{definition}\label{normeta} Let $N\geq 2$. Define the (normalized) {\it $\eta$-weight} $${\bf w}_{N,\eta}(\lambda):=
\left\{
		\begin{array}{ll}
			 \pm u_\eta \displaystyle\frac{\langle 1\rangle \langle N-1\rangle}{\langle 0\rangle^2} & \mbox{if $\lambda^{\textrm{mod }N}=\eta$} \\
			0 & \mbox{otherwise.} 
		\end{array}
	\right.$$
The sign of ${\bf w}_{N,\eta}(\lambda)$, at least for $\eta=\emptyset$, will be clarified in Proposition \ref{pillow1}. The constant $u_\eta$ is the unique one such that ${\bf w}_{N,\eta}(\eta)=1$. Define $${\bf w}_N(\lambda):={\bf w}_{N,\emptyset}(\lambda).$$ \end{definition}

\begin{remark} For $N=2,3,4,6$ the weight ${\bf \tilde{w}}_{N,\emptyset}(\lambda)$ does not require normalization, and already equals ${\bf w}_N(\lambda)$. \end{remark}

\subsection{Generalizations of the Bloch-Okounkov bracket}

To state the main result of this paper, we introduce the notion of a quasimodular form. Let $q=e^{2\pi i\tau}$ with $\tau$ a complex variable in the upper half-plane $\mathbb{H}$.

\begin{definition} The {\it congruence subgroup} $\Gamma(N)$ is the kernel of the reduction mod $N$ homomorphism $SL_2(\Z)\rightarrow SL_2(\Z/N\Z)$. The larger group $\Gamma_1(N)$ are those matrices whose reduction mod $N$ is strictly upper triangular. A holomorphic function $f\,:\,\mathbb{H}\rightarrow \C$ is a {\it modular form of weight $k$} for $\Gamma\subset SL_2(\Z)$ if \begin{enumerate}
\item[(1)] $f(\tfrac{a\tau +b}{c\tau +d})=(c\tau+d)^kf(\tau)$ for all $\twobytwo{a}{b}{c}{d}\in \Gamma$
\item[(2)] $f(\tau)$ has polynomial growth as $\tau$ approaches the cusps of $\mathbb{H}/\Gamma$.
\end{enumerate}
Then, the modular forms for $\Gamma$ form a graded ring $M^*(\Gamma)$, graded by weight. For our purposes, it suffices to define $$QM^*(\Gamma):=M^*(\Gamma)[E_2]$$ where $E_2(\tau):=-\tfrac{1}{24}+\sum \sigma(n)q^n$ is the weight $2$ Eisenstein series. The graded pieces of $QM^*(\Gamma(N))$ are finite-dimensional.

\end{definition}

Our main result, proven in Section \ref{proof}, is:

\begin{theorem}\label{modular1} Let $N\geq 2$. For any ${\bf F}\in \Lambda_N^*$ of pure weight, the $q$-series $$\langle {\bf F}\rangle_{{\bf w}_N}:=\frac{\sum_{\lambda} q^{|\lambda|}{\bf w}_N(\lambda){\bf F}(\lambda)}{\sum_{\lambda} q^{|\lambda|}{\bf w}_N(\lambda)}$$ is the result of substituting $q\mapsto q^N$ into a quasimodular form for $\Gamma_1(N)$ with weight equal to $\textrm{wt}\,\, {\bf F}$. \end{theorem}

The results of this paper verify Theorem \ref{modular1} only for averages of functions against the weight ${\bf w}_N$, rather than the more general weight ${\bf w}_{N,\eta}$. The case of $\eta\neq \emptyset$ is proven in \cite{engel}, thus verifying what was a conjecture in the first version of this paper:

\begin{theorem}\label{conj} Let ${\bf F}\in \Lambda_N^*$. The $q$-series $$\frac{\sum_{\lambda} q^{|\lambda|/N}{\bf w}_{N,\eta}(\lambda){\bf F}(\lambda)}{\sum_{\lambda} q^{|\lambda|/N}{\bf w}_N(\lambda)}$$ is quasi-modular of level $\Gamma(N)$ for any $N$-core $\eta$. \end{theorem}

\begin{remark} The definition of $\langle \cdot \rangle_{{\bf w}_{N,\eta}}$ is a generalization of the Bloch-Okounkov bracket, denoted $\langle \cdot \rangle_q$ in \cite{eo1}. It is the case $N=1$, $\eta=\emptyset$. It is also a generalization of the bracket denoted $\langle \cdot \rangle_{{\bf w}_2}$ featured in \cite{eo2}, which is the $N=2$, $\eta=\emptyset$ case.   \end{remark}

When applying Theorem \ref{conj} to Hurwitz theory, the factor of $$H_N(\emptyset\,|\,\emptyset,\emptyset,\emptyset)=\sum_{\lambda} q^{|\lambda|}{\bf w}_N(\lambda)$$ in the denominator is a count of orbifold-unramified covers of the elliptic orbifold which factor through the map $E\to X_N$.  Every branched cover $\Sigma\to X_N$ decomposes uniquely into two components $$\Sigma=\Sigma^{\rm unr}\,\sqcup\, \Sigma^\bullet$$ with $\Sigma^{\rm unr}$ consisting of all connected components which factor through an unramified cover of $E$, and $\Sigma^\bullet$ its complement. Since the degree of a cover is the sum of the degrees of these two components, dividing by $H_N(\emptyset\,|\,\emptyset,\emptyset,\emptyset)$ restricts to covers for which every component has ``ramification" in the orbifold sense. Thus Theorem \ref{conj} and Proposition \ref{modify} imply:

\begin{proposition} Let $H_N^{\bullet}(D)$ denote the generating function for covers of $X_N$ with ramification profile $D=\{\nu^i\,|\,\mu^a,\mu^b,\mu^c\}$ such that there is ramification on every connected component of the cover. Then $$H_N^{\bullet}(D)=\sum_{N\textrm{-cores }\eta} \langle {\bf F}_\eta^D\rangle_{{\bf w}_{N,\eta}}$$ is in the ring $QM^*(\Gamma(N))$ of quasi-modular forms for $\Gamma(N)$ and has weight bounded in terms of $D$.\end{proposition}

To further pass from covers with ramification on each component to connected covers, there is a standard solution via M\"obius inversion, see e.g. Section 2.3 of \cite{eo1} or Section 5.5 of \cite{goujard}. Essentially, one applies the inclusion-exclusion principle to all the ways in which the curvature profile can split up on the connected components. So we have:

\begin{corollary} The generating function $H_N^{\circ}(D)$ for connected covers of $X_N$ with ramification profile $D$ lies in $QM^*(\Gamma(N))$.  \end{corollary}

By the discussion following Definition \ref{general}, we conclude quasimodularity of generating functions of tilings:

\begin{corollary}\label{modular2} The following generating functions of surface tilings with specified non-zero curvatures lie in a ring of quasimodular forms:  \begin{table}[H]
\begin{tabular}{| c | c | c | c |} \hline 
Tile
& $\,\,\,$Level$\,\,\,$
\\
\hline
triangle
& $\Gamma_1(6)$
\\
hexagon {\rm (}$6$-duck tile{\rm )}
& $\Gamma_1(6)$
\\
square {\rm (}$4$-duck tile{\rm )}
& $\Gamma_1(4)$
\\
vertex bicolored hexagon {\rm (}$3$-duck tile{\rm )}
& $\Gamma_1(3)$
\\
edge bicolored quadrilateral {\rm (}$2$-duck tile{\rm )}
& $\Gamma_1(2)$
\\
asymmetric quadrilateral {\rm (}$1$-duck tile{\rm )}
& $SL_2(\Z)$
\\
\hline
\end{tabular}
\end{table}
\end{corollary}

All but the first case follows from Theorem \ref{modular1}, with only the first case requiring Theorem \ref{conj}. The level, even for triangulations, is $\Gamma_1(6)$ because the generating function is expressible in integer powers of $q$, rather than $q^{1/N}$. So the generating function is modular with respect to $\tau\mapsto \tau+1$.

\begin{remark} Using arithmetic techniques, P. Smillie and the author \cite{engsmi} showed that the appropriate generating function for positive curvature triangulations is in fact a modular form for $SL_2(\Z)$ of pure weight $10$. In particular, the above results are not necessarily optimal---for instance, Corollary \ref{modular2} only says that this generating function is quasi-modular of mixed weight for $\Gamma_1(6)$ with a large weight bound. \end{remark}

The vector space $QM^{\leq{\rm wt}}(\Gamma_1(N))$ is finite-dimensional and even has explicit bases. Thus it is possible in principle to compute the generating function for any of the above tiling problems from the knowledge of a finite number of $q$-coefficients.

The following generating function encodes all ${\bf w}_N$ brackets:

\begin{definition} The {\it $n$-point function} is $$F_N(z_1,\dots,z_n):=\left\langle \,\prod_{i=1}^n {\bf e}(\lambda,\ln z_i)\right\rangle_{{\bf w}_N}.$$ \end{definition}

The $n$-point functions determine $\left\langle \prod {\bf p}^{r_i}_{k_i}\right\rangle_{{\bf w}_N}$ for any $k_i\in\N$ and $r_i\textrm{ mod }N$. Up to a constant, these brackets are coefficients in the Taylor expansion of $F_N(e^{z_1},\dots,e^{z_n})$ about the point $(z_1,\dots,z_n)=(\tfrac{2\pi i r_1}{N},\dots,\tfrac{2\pi i r_n}{N}).$

\section{Fock space and vertex operators}\label{proof}

 In this section, we prove Theorem \ref{modular1} using the Fock space formalism.

\subsection{More operators on Fock space} In the previous section, we have defined a number of functions on partitions, in particular the functions ${\bf p}_k^r(\lambda)$ which freely generate the cyclotomic enlargement $\Lambda_N^*$ of the ring of shifted-symmetric functions. We now promote these functions to operators on Fock space by defining $$\mathscr{P}^r_k v_\lambda={\bf p}_k^r(\lambda)v_\lambda$$ as an operator acting diagonally in the $v_\lambda$ basis of the charge zero subspace $\mathbb{L}=\Lambda^{\infty/2}_0V$. These operators are unbounded, but are bounded on the $v_\lambda$ with $|\lambda|$ fixed. As for ${\bf p}_k^r(\lambda)$, the operators $\mathscr{P}_k^r$ are themselves Taylor coefficients of an operator which depends on an analytic variable $z$.

Define for any function $f:\R\rightarrow \R$ and an integer $k$ an operator on $V$ by $\underline{i}\rightarrow f(i)\,\underline{i-k}$. This induces an operator on Fock space \begin{align*}\mathscr{E}_k[f]\,:\,\underline{i_1}\wedge \underline{i_2}\wedge \dots\,\mapsto\, & f(i_1)\cdot \underline{i_1-k}\wedge \underline{i_2}\wedge \dots+ \\ &f(i_2)\cdot \underline{i_1}\wedge \underline{i_2-k}\wedge \dots+ \dots \end{align*} which must regularized when $k=0$. When $f$ is a quasi-polynomial of period $N$, one regularizes $\mathcal{E}_0[f]$ via $L$-function values, whereas when $f$ is an exponential function, one can regularize via the geometric series.

Define $\mathcal{E}_k(z):=\mathcal{E}_k[e^{zx}]$ where $z$ is a formal variable. Then $$\mathcal{E}_0(z)v_\lambda = {\bf e}(\lambda,z)v_\lambda$$ acts diagonally by the eigenvalue ${\bf e}(\lambda,z)=\sum e^{(\lambda_i-i+\frac{1}{2})z}$ depending analytically on $z$. It follows from the definition of ${\bf p}_k^r$ that $$\mathscr{P}^r_k=\zeta_N^{r/2}k![z^k]\mathcal{E}_0(z+\tfrac{2\pi ri}{N}).$$ Recall that the energy operator acts by $H v_\lambda:=|\lambda|v_\lambda$. It satisfies $$H=\mathscr{P}^0_1+\frac{1}{24}.$$ We call $\mathcal{E}_k(z)$ a {\it vertex operator}. The vertex operators also encode the action of the Heisenberg algebra on Fock space: $$\alpha_n=\mathcal{E}_n(0).$$ We say an operator expressed in terms of the $\alpha_n$ is {\it normal-ordered} if all raising operators, i.e. those involving $\alpha_n$ for $n$ negative, appear to the left of all lowering operators, i.e. those involving $\alpha_n$ for $n$ positive.

\begin{proposition}\label{pillow1} Let $N\geq 2$. Define the normal-ordered operator $$\mathfrak{W}_N:=\exp\left(\sum_{n\in \Z \backslash N\Z}\!\frac{\alpha_n}{n}\right).$$ Then the diagonal entries of $\mathfrak{W}_N$ are $$\langle v_\lambda\,\big{|}\, \mathfrak{W}_N\,\big{|}\,v_\lambda\rangle={
	\left\{
		\begin{array}{ll}
			\pm \displaystyle\frac{\langle 1\rangle\langle N-1\rangle}{\langle 0\rangle^2} & \textrm{if }\lambda\textrm{ is }N\textrm{-decomposable}  \\
			0 & \mbox{otherwise}
		\end{array}
	\right.
}$$ where $\langle a\rangle$ is the product of the hooklengths of $\lambda$ congruent to $a$ mod $N$. \end{proposition}

\begin{proof} The proof is analogous to the $N=2$ case in \cite{eo2}, where this operator was dubbed the pillowcase operator. Choose $\ell$ such that $\lambda_i=0$ for all $i>N\ell$. Since $\mathfrak{W}_N$ is the product of a lower unitriangular and upper unitriangular operator, $\langle v_\lambda\,\big{|}\, \mathfrak{W}_N\,\big{|}\,v_\lambda\rangle$ can be computed in the wedge product of a finite-dimensional truncation $V^{[\ell]}$ with basis $$e_k=\underline{-N\ell+k+\tfrac{1}{2}},\,\,\,\,\,\,\,\,\,k=0,\,\dots,\,\lambda_1+N\ell-1.$$ The matrix elements of $\mathfrak{W}_N$ are then determinants of a minor of matrix elements of the action on $V^{[\ell]}$ itself. These matrix elements are computed in the following lemma:

\begin{lemma}\label{pillow2} Let $f(x,y)$ be the generating function for matrix entries of $\mathfrak{W}_N$ acting on $V^{[\ell]}$. Then $$f(x,y):=\sum_{k,l}\,\langle e_k\,\big{|}\,\mathfrak{W}_N\,\big{|}\,e_l\rangle x^ky^l=\frac{1}{1-xy}\cdot \frac{1-x}{(1-x^N)^{1/N}}\cdot \frac{(1-y^N)^{1/N}}{1-y} .$$ More explicitly, we have for $N\geq 3$ $$\frac{\langle e_k\,\big{|}\,\mathfrak{W}_N\,\big{|}\,e_l\rangle}{\mathfrak{b}(k)\mathfrak{c}(l)}=\threepartdef{0}{k\not\equiv 0\textrm{ and }k\not\equiv l+1\textrm{ mod }N}{1}{k\equiv 0\textrm{ and }k\not\equiv l+1\textrm{ mod }N}{(l-k)^{-1}}{k\equiv l+1\textrm{ mod }N}$$ where $$\mathfrak{b}(k):=\prod_{r\leq k} \frac{\{r\equiv -1\textrm{ mod }N\}}{\{r\equiv 0\textrm{ mod }N\}}\,\,\,\,\,\,\,\,\textrm{and}\,\,\,\,\,\,\,\,\mathfrak{c}(l):=\prod_{r\leq l} \frac{\{r\equiv 1\textrm{ mod }N\}}{\{r\equiv 0\textrm{ mod }N}.$$ \end{lemma}

\begin{proof} On $V^{[\ell]}$, the action of the raising and lowering operators is $$\alpha_n\cdot e_k=\alpha_1^n\cdot e_k=\twopartdefotherwise{e_{k-n}}{k-n\in\{0,\dots,\lambda_1+N\ell-1\}}{0}$$ Define the operator $$\mathfrak{B}_N:=\exp\left(\sum_{n\in \N\backslash N\N} \frac{\alpha_n}{n}\right)$$ and observe that $$\exp\left(\sum_{n\in\N\backslash N\N}\frac{x^n}{n}\right)=\frac{(1-x^N)^{1/N}}{1-x}.$$ Since $\alpha_n$ and $\alpha_{-n}$ are adjoint, $\mathfrak{W}_N=(\mathfrak{B}_N^*)^{-1}\mathfrak{B}_N$. By expanding $\mathfrak{B}_N^{-1}e_k$ and $\mathfrak{B}_N e_l$ in the orthonormal basis $\{e_m\}$ we compute that \begin{align*}\langle \mathfrak{B}_N^{-1}e_k\,\big{|}\,\mathfrak{B}_Ne_l\rangle & =\sum_{m\geq 0} \langle \mathfrak{B}_N^{-1}e_k \,\big{|}\,e_m \rangle \cdot \langle e_m\,\big{|}\,\mathfrak{B}_Ne_l\rangle \\ &=\sum_{m\geq 0}\, [x^{k-m}y^{l-m}] \frac{1-x}{(1-x^N)^{1/N}}\cdot \frac{(1-y^N)^{1/N}}{1-y} \\ &= [x^ky^l]\frac{1}{1-xy}\cdot \frac{1-x}{(1-x^N)^{1/N}}\cdot \frac{(1-y^N)^{1/N}}{1-y} \end{align*} from which the first claim follows. Expanding via the binomial theorem and geometric series, we have $$f(x,y)=\sum_{a,b,c\geq 0}(-1)^{a+b}{-1/N\choose a}{1/N\choose b}(x^{aN}-x^{aN+c+1})y^{bN+c}.$$ Induction or manipulation of binomial coefficients verifies the formula for the $x^ky^l$ coefficient of the above expression. 
\end{proof}

\begin{lemma}\label{pillow3} We have $\langle v_\lambda\,\big{|}\, \mathfrak{W}_N\,\big{|}\, v_\lambda\rangle=0$ unless $\lambda$ is $N$-decomposable. \end{lemma}

\begin{proof} By general properties of wedge products, we can evaluate $\langle v_\lambda \,\big{|}\, \mathfrak{W}_N\,\big{|}\, v_\lambda\rangle$ by taking the determinant of $$\langle e_k\,\big{|}\,\mathfrak{W}_N\,\big{|}\, e_l\rangle_{k,l\in\{\lambda_i-i+N\ell\}}.$$ This determinant is necessarily stable upon appending zeroes to $\lambda$. We do so until $N$ divides $\ell(\lambda)$, say $\ell(\lambda)=Nd$. By Lemma \ref{pillow2}, the above matrix has the form of an $Nd \times Nd$ block matrix with many zero blocks: $$\begin{pmatrix} * &  *  & \cdots & * & * \\ * & 0 & \cdots & 0 & 0 \\ 0 & * & \cdots & 0 & 0 \\ \vdots  & \vdots & & \vdots & \vdots \\ 0  & 0 & \cdots & * & 0\end{pmatrix}.$$ We index the blocks starting from zero so that the $(r,s)$-block consists of the entries such that $k\equiv r\textrm{ mod }N$ and $l\equiv s\textrm{ mod }N$. Observe that this requires reordering the $e_k$.

Let $n_r$ be the width of the $r$th row. We first remark that for $0\leq r\leq N-2$, the determinant vanishes unless $$n_r-1 \leq n_{r+1} \leq n_r.$$ The upper bound is immediate because the $(r+1,r)$-block is the only non-zero entry in the $(r+1)$th row. On the other hand, the only other non-zero entry in the $r$th column is the $(0,r)$-block, which has rank one. This gives the lower bound. Similar logic applied to the $(0,N-1)$-block implies that $$n_{N-1}\leq n_0\leq n_{N-1}+1.$$ Next, observe there is at most one $r\in\{0,\dots,N-2\}$ such that $n_r>n_{r+1}$ as otherwise the $(0,N-1)$ block cannot have the correct proportions. Therefore $$Nn_0\geq \sum n_r \geq Nn_0-(N-1).$$ To finish the proof, we observe that $\sum n_r=Nd$ is divisible by $N$ and thus $n_r=n_{r+1}$ for all $i$. In particular, the sets $\{i\,\big{|}\,\lambda_i-i\equiv r\textrm{ mod }N\}$ have equal size for all $r$ and thus $\lambda$ is $N$-decomposable. \end{proof}

Returning to the proof of Proposition \ref{pillow1}, we have verified the formula when $\lambda$ is not $N$-decomposable. So suppose $\lambda$ is $N$-decomposable. Then the determinant of the matrix in Lemma \ref{pillow3} factors into the product of the determinants of the $(r+1,r)$-blocks. By Lemma \ref{pillow2}, \begin{align*}\langle v_\lambda\,\big{|}\, \mathfrak{W}_N\,\big{|}\, v_\lambda\rangle&=\prod_{i=1}^{N\ell} \mathfrak{b}(\lambda_i-i+N\ell)\mathfrak{c}(\lambda_i-i+N\ell) \cdot \\ &\hspace{70pt}\prod_{r\textrm{ mod }N} \!\!\!\det((\lambda_l-\lambda_k+k-l)^{-1})_{\lambda_k-k\equiv r+1}^{\lambda_l -l\equiv r}.\end{align*}

We now apply the Cauchy determinant formula $$\det\left(\frac{1}{x_k-y_l}\right)=\frac{\prod_{k<l} (x_k-x_l)(y_l-y_k)}{\prod_{k,l} (x_k-y_l)}$$ to conclude that
 \begin{align*} \langle v_\lambda\,\big{|}\,\mathfrak{W}_N\,\big{|}\,v_\lambda\rangle&=\pm \!\!\!\!\!\!\!\!\!\!\!\!\prod_{\hspace{15pt}r\leq \lambda_i-i+N\ell}\!\!\!\!\!\!\!\!\!\!\!\!\!\prod\,\,\,\frac{(r\equiv \pm 1\textrm{ mod }N)}{(r\equiv 0\textrm{ mod }N)^2} \frac{\prod\{\lambda_i-\lambda_j+j-i\equiv 0\textrm{ mod }N\}^2}{\prod\{\lambda_i-\lambda_j+j-i\equiv \pm 1\textrm{ mod }N\}}.\end{align*} Finally, we have $$\#\{i\,\big{|}\,r\leq \lambda_i-i+N\ell\}-\#\{(i,j)\,\big{|}\,\lambda_i-\lambda_j+j-i=r\}=\#\{r\textrm{-hooks}\}$$ from which the proposition follows. \end{proof}

\begin{remark}\label{sign} To continue our arguments, we ought to verify that the sign in Proposition \ref{pillow1} agrees with the sign of ${\bf w}_N(\lambda)$ as determined by Proposition \ref{hooks} for $N=1,2,3,4,6$. This can be proven by induction on $|\lambda|$---one removes an $N$-rim hook and checks that both quantities change by the same sign. As noted in Definition \ref{normeta}, the sign of ${\bf w}_N(\lambda)$ was left ambiguous for values of $N\neq 1,2,3,4,6$. We may define the sign so that Proposition \ref{pillow1} remains true for all values of $N$.  \end{remark}

\subsection{Determination of the $n$-point function} Proposition \ref{pillow1} and Remark \ref{sign} imply the following formula for the $n$-point function:
\begin{align}\label{npoint1} F_N(x_1,\dots,x_n)=\frac{tr_{\mathbb{L}}\,q^H\mathcal{E}_0(\ln x_1)\dots\mathcal{E}_0(\ln x_n)\, \mathfrak{W}_N}{tr_{\mathbb{L}}\,q^H\mathfrak{W}_N}.\end{align} The trace of this product of operators on $\mathbb{L}$ should be thought of as valued in the power series ring $\C[[q]]$, with the $q^n$ coefficient equal to the contribution from the energy eigenspace $H=n$.  In \cite{eo2}, many details of the evaluation of this trace are left to the reader, and so we somewhat expand their arguments. See also Section 3 of \cite{milas} for an analogous treatment of the $N=1$ case. The key tool is the {\it boson-fermion correspondence}, which can be phrased as the decomposition of the charge zero subspace of Fock space with respect to the action of $\alpha_{-m}$ for all $m> 0$: \begin{align*}\bigotimes_{m=1}^\infty \bigoplus_{k=0}^\infty \alpha_{-m}^k  v_\emptyset &\xrightarrow{\sim}\mathbb{L} \\ \alpha_{-m_1}^{k_1}v_\emptyset \otimes \cdots \otimes \alpha_{m_r}^{k_r}v_\emptyset &\mapsto  \left(\prod \alpha_{-m_i}^{k_i}\right)v_\emptyset. \end{align*} 

The boson-fermion correspondence is the central idea for proving quasimodularity of generating functions of Hurwitz numbers: One expresses the desired series as a $q$-trace in the ``fermionic" $v_\lambda$ basis, then computes the trace in the above ``bosonic" basis. To this end, Eskin-Okounkov \cite{eo2} (see also \cite{kac}, Theorem 14.10) give the following formula: \begin{align}\label{bosonic}\mathcal{E}_0(\ln x)=[y^0]\frac{1}{x^{1/2}-x^{-1/2}}\exp\left(\,\sum_{n\in\Z\backslash 0} \frac{(xy)^n-y^n}{n}\alpha_n\right),\end{align} where $[y^0]$ means the coefficient of $y^0$ in a Laurent series expansion. The operators $\mathcal{E}_0(\ln x_i)$ and $\mathfrak{W}_N$ are tensor products over $m$ of operators of the form $\exp(A\alpha_{-m}/m)\exp(B\alpha_m/m)$ acting on the tensor factor $$\mathbb{L}_m:=\bigoplus_{k=0}^\infty \alpha_{-m}^kv_\emptyset$$ for various constants $A$ and $B$. Thus, the trace (\ref{npoint1}) factors into a product over $m\geq 1$. We are led to analyze for each $m$ the following product of operators:

\begin{lemma}\label{tracelemma} For $A_i$ and $B_i$ making the expressions convergent, we have \begin{align*}\label{trace} q^H e^{A_1\alpha_{-m}}e^{B_1\alpha_m}\dots e^{A_n\alpha_{-m}}e^{B_n\alpha_m}=q^He^{(\sum_i\! A_i)\alpha_{-m}}e^{(\sum_j\! B_j)\alpha_m}\prod_{i< j} e^{B_iA_jm}.\end{align*} \end{lemma}

\begin{proof} The operator is easy to normal-order---the commutation relation $[B\alpha_m,A\alpha_{-m}]=BAm$ implies that $e^{B\alpha_m}e^{A\alpha_{-m}}=e^{A\alpha_{-m}}e^{B\alpha_m}e^{ABm}.$ We iteratively use this relation to move every term $e^{B_j\alpha_m}$ to the right. \end{proof}

\begin{remark}\label{convergence} Substituting (\ref{bosonic}) into (\ref{npoint1}) gives an expression which is convergent whenever $|qy_n| <|x_1y_1|<|y_1|<\cdots<|x_ny_n|<|y_n|<1.$ The arguments are the same as Section 3.2.2 of \cite{eo2}, though our convention differs slightly as we work inside the unit disc rather than outside it. This arises from (\ref{bosonic}) being expressed in terms of the transpose of the operator used in \cite{eo2}, which is irrelevant because $\mathcal{E}_0(\ln x)$ is symmetric. \end{remark}

We can now combine (\ref{npoint1}), (\ref{bosonic}), Lemma \ref{tracelemma}, and the formula $$\textrm{tr}_{\mathbb{L}_m}\,\,q^He^{A\alpha_{-m}}e^{B\alpha_m}=\frac{1}{1-q^m}\exp\left(\frac{ABmq^m}{1-q^m}\right)$$ to determine an expression for the $n$-point function. The expression is huge, so we incorporate some simplifying steps before writing it. We remark that the terms $e^{B_iA_jm}$ independent of $q$ combine over all $m$ to give $$\prod_{m\geq 1}\exp(\pm C^m/m)=(1-C)^{\mp 1}$$ for various constants $C$. Furthermore, the terms involving $q$ are of the form $$\prod_{m\geq 1} \exp\left(\frac{\pm (Cq)^m}{m(1-q^m)}\right)=\prod_{m\geq 1} (1-Cq^m)^{\mp 1},$$ which is seen by expanding $(1-q^m)^{-1}$ in a geometric series. So the resulting expression of the $n$-point function is a product of rational functions: \begin{align} \begin{aligned} &\,[y_0^0\dots y_n^0]\prod_i (x_i^{1/2}-x_i^{-1/2})^{-1} \prod_{i<j} \frac{(1-x_iy_i/x_jy_j)(1-y_i/y_j)}{(1-y_i/x_jy_j)(1-x_iy_i/y_j)} \cdot \\ &\prod_{m\geq 1}\prod_{i,j} \frac{(1-x_iy_i/x_jy_jq^m)(1-y_i/y_jq^m)}{(1-y_i/x_jy_jq^m)(1-x_iy_i/y_jq^m)} \cdot \\ & \prod_i  \frac{1-x_iy_i}{1-y_i}\prod_{m\geq 1}\frac{(1-x_iy_iq^m)(1-q^m/x_iy_i)}{(1-y_iq^m)(1-q^m/y_i)}\cdot \\ &\left[\prod_i \frac{1-y_i^N}{1-x_i^Ny_i^N}\prod_{m\geq 1} \frac{(1-y_i^Nq^{Nm})(1-q^{Nm}/y_i^N)}{(1-x_i^Ny_i^Nq^{Nm})(1-q^{Nm}/x_i^Ny_i^N)}\right]^{1/N}.\end{aligned}\end{align} As the operator $\mathfrak{W}_N$ is exp of an operator indexed by integers not divisible by $N$, we incorporate and then extract terms indexed by $N\Z\backslash \{0\}$, which gives the third and fourth lines in the above formula.

Next, we collect terms using the Jacobi triple product formula: $$j(x,q):=\sum_{n\in\Z} (-x)^nq^{\frac{n^2-n}{2}}=\prod_{m\geq 1} (1-q^m)(1-xq^{m-1})(1-x^{-1}q^m).$$ Define the theta function to be $$\vartheta(x,q):=x^{-1/2}\frac{j(x,q)}{j'(1,q)}=(x^{1/2}-x^{-1/2})\prod_{m\geq 1}\frac{(1-xq^m)(1-x^{-1}q^m)}{(1-q^m)^2}$$ which satisfies the following transformation rules: \begin{align*} \vartheta(x^{-1},q)& =-\vartheta(x,q) \\ \vartheta(e^{2\pi i}x,q)&=-\vartheta(x,q) \\ \vartheta(qx,q)&=-q^{-1/2}x^{-1}\vartheta(x,q)\end{align*} in addition to $\vartheta'(1)=1$. The second equality is made sensical by observing that $\vartheta$ is two-valued on $x\in\C^*$. Note that $$\frac{\vartheta(u^N,q^N)}{\vartheta(v^N,q^N)}=\prod_{r\textrm{ mod }N}\frac{\vartheta(\zeta_N^ru,q)}{\vartheta (\zeta_N^rv,q)}.$$ Let $x=(x_0,\dots,x_n)$ and $y=(y_0,\dots,y_n)$. Simplifying (8) then gives

\begin{proposition} In the domain where the trace converges, see Remark \ref{convergence}, the $n$-point function is $F_N(x)=[y_1^0\dots y_n^0] G_N(x,y)$ where \begin{align}\label{fp} G_N(x,y):=\prod_i \frac{\vartheta(x_iy_i)}{\vartheta(x_i)\vartheta(y_i)} \prod_{i<j} \frac{\vartheta(y_i/y_j)\vartheta(x_iy_i/x_jy_j)}{\vartheta(x_iy_i/y_j)\vartheta(y_i/x_jy_j)} \prod_{i,r} \frac{\vartheta(\zeta_N^ry_i)^{1/N}}{\vartheta(\zeta_N^r x_iy_i)^{1/N}}.\end{align} \end{proposition}

As the second argument is always $q$ in the above theta functions, we have suppressed it. We extract the $y_1^0\dots y_n^0$ coefficient of $G_N(x,y)$ via contour integration about a product of circles $|y_i|=c_i$ in the domain of Remark \ref{convergence}. Setting $d\mu_y=(2\pi i)^{-n}\prod dy_i/y_i$ we then have \begin{align}\label{oint} F_N(x)=\displaystyle \oint_{|y_n|=c_n}\!\!\!\!\!\!\!\cdots\oint_{|y_1|=c_1}\!\!\!\!\!\! G_N(x,y)\,d\mu_y.\end{align}

\begin{remark}\label{symmetry} Let $\sigma\in S_n$ be a permutation, and let $x^\sigma$ and $y^\sigma$ be the result of permuting the indices of the variables. From (\ref{fp}), we see that $$G_N(x^\sigma,y^\sigma)=G_N(x,y).$$ In addition, we know after the fact that $F_N(x^\sigma)=F_N(x)$ is symmetric. Denote the contour of (\ref{oint}) by $R$. Then \begin{align*} F_N(x)=& F_N(x^\sigma)=\oint_R G_N(x^\sigma,y)\,d\mu_y \\ &= \oint_R G_N(x,y^{\sigma^{-1}})\,d\mu_y = \oint_{\sigma^{-1}R} G_N(x,y)\,d\mu_y \end{align*} where the last equality is the change-of-variables formula. We conclude that formula (\ref{oint}) still holds even if we permute the ordering of the values $c_i$. \end{remark}

\begin{remark}\label{bielliptic} 
Consider the universal family of degenerating elliptic curves $\mathcal{E}\rightarrow \Delta_q$ over an analytic disc whose fiber over $q\neq 0$ is $\C^*/q^\Z$. Call this family the {\it Tate curve}. Let $$\mathcal{E}^{\times n}\times \mathcal{E}^{\times n}\rightarrow \Delta_q$$ be the $2n$-fold fiber product of the Tate curve with itself. Let $x,y\in (\C^*)^n$ be fiber coordinates on the first $n$ and second $n$ factors respectively. The transformation laws for $\vartheta$ imply that $G_N(x,y)$ is invariant, {\it as an $N$-multivalued function}, under the substitutions $y_j\mapsto e^{2\pi i }y_j$ and $y_j\mapsto e^{2\pi i \tau}y_j$ whereas \begin{align*} G_N(x,y)\big{|}_{x_j\mapsto e^{2\pi i}x_j} & = -G_N(x,y) \\ G_N(x,y)\big{|}_{x_j\mapsto qx_j} & =-q^{1/2}x_j\frac{x_{j+1}\cdots x_n}{x_1\cdots x_{j-1}}G_N(x,y). \end{align*}

Following Bloch-Okounkov \cite{bo}, it's natural to define: $$H_N(x,y):=\vartheta(x_1\cdots x_n)G_N(x,y).$$ Then $H_N(x,y)$ is invariant under $x_j\mapsto e^{2\pi i}x_j$ and has factor of automorphy $(x_1\cdots x_{j-1})^{-2}$ under the substitution $x_j\mapsto e^{2\pi i \tau} x_j$. So $H_N(x,y)$ is a multisection of a natural line bundle on $\mathcal{E}^{\times n}\times \mathcal{E}^{\times n}$. It may be possible to find an explicit formula for $F_N(x)$ analogous to the Theorem 6.1 from \cite{bo}. \end{remark}

\subsection{Extraction of Taylor coefficients} We now evaluate (\ref{oint}), again following \cite{eo2}. By making $|x_i|$ very close to $1$ and $|q|$ very small, we may assume $$|y_i|> |y_j|\prod |x_i|^{\pm 100}> |qy_i| \textrm{     for     } i>j.$$ This allows us to successively integrate in each variable $y_i$ without significantly moving the pole loci in the $y_j$ variables. 
We wish to evaluate $$\left\langle \prod {\bf p}^{r_i}_{k_i}\right\rangle_{{\bf w}_N}=\left(\prod \zeta_N^{r_i/2}k_i!\right) \,[z_1^{k_1}\dots z_n^{k_n}]F_N(\zeta_N^{r_1}e^{z_1},\dots,\zeta_N^{r_n}e^{z_n}).$$ It is thus convenient to change coordinates by setting $x_i=e^{z_i}$ and $y_i=e^{w_i}$ and to define $\Theta(u,q):=\vartheta(e^u,q)$ so that $u\in\C$. 
Define \begin{align*} F_N(e^z)&:=F_N(e^{z_1},\dots,e^{z_n}) \\  G_N(e^z,e^w)&:= G_N(e^{z_1},\dots,e^{z_n},e^{w_1},\dots,e^{w_n}). \end{align*} Changing variables, we have $$F_N(e^z)=\frac{1}{(2\pi i)^n}\oiint_{|e^{w_i}|=c_i} G_N(e^z,e^w) \,dw.$$ This integral appears difficult to compute directly, but as we are ultimately interested in the Taylor coefficients of $F_N(e^z)$, we consider the expansion of $G_N(e^z,e^w)$ about $z_i=2\pi r_ii/N$: $$G_N(e^z,e^w)=\prod_i \frac{1}{\Theta(z_i)} \sum_{k_i\geq 0} g_{k_i}^{r_i}(w_1,\dots,w_n) \prod_i (z_i-2\pi r_ii/N)^{k_i}.$$ We do not expand $\Theta(z_i)^{-1}$ because it may introduce a pole at $z_i=0$.

By Remark \ref{bielliptic}, we have that $G_N(e^z,e^w)$ is elliptic i.e. biperiodic in the $w_i$ variables, up to an $N$th root of unity. Considering the constant term of the above expansion shows that \begin{align} \label{trans} g_{k_i}^{r_i}(w_1,\dots,w_n)\big{|}_{w_i\mapsto w_i+2\pi i \tau}= \zeta_N^{-r_i}g_{k_i}^{r_i}(w_1,\dots,w_n).\end{align}

\begin{definition} A {\it theta ratio $f(w)$ of weight $k$} is a ratio $$f(w)=\prod_i \frac{\Theta^{(n_i)}(w-a_i)}{\Theta(w-b_i)}\cdot \frac{\prod_j\Theta^{(s_j)}(c_j)}{\prod_l \Theta(d_l)^{t_l}}$$ where $a_i$, $b_i$, $c_j$, $d_l$ are constants with respect to $w$, such that \begin{align*} k&=\sum n_i +\sum (s_j-1) +\sum t_l.\end{align*} \end{definition}

The derivative of a theta ratio of weight $k$ is a linear combination of theta ratios of weight $k+1$. The terms $g_{k_i}^{r_i}$ are computed by factoring out $\prod \Theta(z_i)^{-1}$, then taking derivatives of the remaining part of (\ref{fp}) and evaluating at $z_i=2\pi r_ii/N$. The result is a ratio of theta functions and derivatives, evaluated at $w_i+\frac{2\pi ri}{N}$ and $w_i-w_j+\frac{2\pi ri}{N}$. In fact, $g_{k_i}^{r_i}$ is independently a theta ratio in each of the $w_i$ variables. For integration of theta ratios, we have Fact 4 from \cite{eo2}:

\begin{lemma}\label{fact} Let $f(w)$ be a theta ratio of weight $k$. Let $C:=e^{\sum (a_i-b_i)}$. Then, the contour integral $$\frac{1}{2\pi i}\oint_{e^w=c} \!\!f(w) \,dw$$ is a linear combination of theta ratios with top weight $k-1$ if $C\neq 1$ and with top weight $k$ if $C = 1$. \end{lemma}

We now successively apply Lemma \ref{fact} to integrate $g_{k_i}^{r_i}$ against each of the $w_i$ variables. Each integral is still a theta ratio in the remaining variables, and by the transformation rule (\ref{trans}), the top weight in the theta ratio decreases by $1$ when $r_i\neq 0$, otherwise the top weight does not decrease.

\begin{lemma} The $n$-point function can be expressed as $$F_N(e^z)=\sum_{k_i\geq -1} a_{k_i}^{r_i}(q) \prod_i (z_i-2\pi r_ii/N)^{k_i}$$ where $a_{k_i}^{r_i}(q)$ is a linear combination of ratios $$\frac{\prod_{r,j} \Theta^{(n_{r,j})}(\tfrac{2\pi ri}{N})}{\prod_r \Theta(\tfrac{2\pi ri}{N})^{m_r}}.$$ Furthermore, $a_{k_i}^{r_i}(q)$ is of pure weight $$\sum k_i+ \#\{i\,\big{|}\,r_i=0\}.$$ \end{lemma}

\begin{proof} We must index by $k_i\geq -1$ because we have incorporated the expansion of $\prod_i \Theta(z_i)^{-1}$, which has a pole along $z_i=0$. The weight computation is as follows: The coefficients of $\prod_i \Theta(z_i)^{-1}$ have weight $n+\sum k_i.$ The weight of $g_{k_i}^{r_i}$ is $\sum k_i$ and each integration either decreases the top weight by $1$ if $r_i\neq 0$ or keeps it constant if $r_i=0$. So the integral of $g_{k_i}^{r_i}$ has top weight $$\sum k_i+ \#\{i\,\big{|}\,r_i=0\}-n.$$ To complete the proof of the lemma, we must show that $a_{k_i}^{r_i}(q)$ is of pure weight. By Remark \ref{symmetry}, we may integrate $G_N(e^z,e^w)$ with respect to any ordering of the $c_i$ and produce the same result. By uniqueness of the Taylor expansion, the same holds true for each $g_{k_i}^{r_i}(w_1,\dots,w_n)$. We may now apply (a mild generalization of) Theorem 7 from \cite{oberdieck}---the symmetrization of the integral of a multivariate elliptic function with respect to all reorderings of the contours is of pure weight. The lemma follows. \end{proof}

\begin{remark} This lemma invalidates the discussion immediately following Theorem 3 of \cite{eo2}, and ultimately proves the numerical observation in \cite{goujard}, that the brackets $\langle \prod {\bf p}_{k_i}^{r_i}\rangle_{{\bf w}_N}$ are pure weight. \end{remark}

Justifying our definition of weight are the following expansions: \begin{align*} \ln \frac{z}{\Theta(z)}&=\sum_{k\geq 2}\frac{z^{2k}}{(2k)!}E_{2k}(q) & \\ \ln \frac{\Theta(\tfrac{2\pi ri}{N})}{\Theta(z+\tfrac{2\pi ri}{N})}&= \sum_{k\geq 1} \frac{z^k}{k!}E_k^r(q) \end{align*} where for $k\geq 3$ and $r\neq 0$, $$E_k^r(q) = (-1)^k(k-1)!\sum_{\lambda\in\Lambda} \frac{1}{(\lambda+\tfrac{2\pi ri}{N})^k}$$ with $\Lambda =2\pi i\Z\oplus 2\pi i \tau\Z$. The Eisenstein series $E_k^r(q)$ is easily seen to be a modular form of weight $k$ for $\Gamma_1(N)$, the subgroup of $SL(\Lambda)$ which preserves the $N$-torsion point $\frac{2\pi ri}{N}\pmod\Lambda$. When $k=1$ or $k=2$, we also have that $E_k^r(q)$ is an Eisenstein series for $\Gamma_1(N)$ of weight $k$. 
If $N=2,3,4,6$ the theta function at $\frac{2\pi i}{N}$ has a special value:
\begin{align*}
\frac{\Theta(\tfrac{2\pi i}{2})}{\zeta_4-\zeta_4^{-1}}&=\prod_{m\geq 1}\frac{(1+q^m)^2}{(1-q^m)^2}=\frac{\eta(q^2)^2}{\eta(q)^4} \\
\frac{\Theta(\tfrac{2\pi i}{3})}{\zeta_6-\zeta_6^{-1}}&=\prod_{m\geq 1}\frac{1+q^m+q^{2m}}{(1-q^m)^2}=\frac{\eta(q^3)}{\eta(q)^3} \\
\frac{\Theta(\tfrac{2\pi i}{4})}{\zeta_8-\zeta_8^{-1}}&=\prod_{m\geq 1}\frac{1+q^{2m}}{(1-q^m)^2}=\frac{\eta(q^4)}{\eta(q)^2\eta(q^2)} \\
\frac{\Theta(\tfrac{2\pi i}{6})}{\zeta_{12}-\zeta_{12}^{-1}}&=\prod_{m\geq 1}\frac{1-q^m+q^{2m}}{(1-q^m)^2}=\frac{\eta(q^6)}{\eta(q)\eta(q^2)\eta(q^3)}.\end{align*}

Suppose $N$ is even. Consider a theta ratio whose arguments are linear functions in the $w_i$ plus some $\frac{2\pi ri}{N}$. Define the {\it parity} of the theta ratio to be the parity of the number of theta functions in the numerator and denominator with $r$ odd. The parity of $g_{k_i}^{r_i}$ is $\#\{i\,\big{|}\,r_i\textrm{ odd}\}$, and each application of Lemma \ref{fact} maintains the parity. The coefficients of the expansion of $\prod \Theta(z_i)^{-1}$ also have parity $\#\{i\,\big{|}\,r_i\textrm{ odd}\}$, and thus $a_{k_i}^{r_i}(q)$ has parity zero. Combining with the above expansions, we therefore conclude:

\begin{lemma}\label{degree} For all $k_i$ and $r_i$, we have $$a_{k_i}^{r_i}(q)\in \Q(\zeta_N)[E_k^r(q),\Theta (\tfrac{2\pi r i}{N})^{\pm 1}].$$ When $N$ is even, the total degree in $\Theta (\tfrac{2\pi r i}{N})$ ranging over $r$ odd is even for any component of $a_{k_i}^{r_i}(q)$. \end{lemma}

We need not work over $\Q(\zeta_{2N})$ when $N$ is even---the parity condition on $a_{k_i}^{r_i}(q)$ implies that it is expressible over $\Q(\zeta_N)$ in terms of Eisenstein series and the above $\eta$-ratios, or more generally $q$-series with coefficients in $\Q(\zeta_N)$. When $N$ is odd, we have $\Q(\zeta_{2N})=\Q(\zeta_N)$ in any case.

Corollary 1 of \cite{yang} implies that a monomial in $\Theta(\frac{2\pi ri}{N})$, with even degree in odd $r$ whenever $N$ is even, is a modular function of level $\Gamma_1(N^2)$. Also, note that $a_{k_i}^{r_i}(q)$ is in fact a power series in $q^N$ because ${\bf w}_N(\lambda)=0$ unless $\lambda$ is $N$-decomposable. Hence, $a_{k_i}^{r_i}(q)$ is modular with respect to $$\twobytwo{1}{0}{0}{N} \Gamma_1(N) \twobytwo{1}{0}{0}{N}^{-1}\supset \Gamma_1(N^2).$$ That is, $a_{k_i}^{r_i}(q)$ is the result of substituting $q\mapsto q^N$ into a weakly quasimodular form for $\Gamma_1(N)$.

Finally, we claim that in fact $a_{k_i}^{r_i}(q)$ is holomorphic at the cusps, and thus a modular form. It is automatically holomorphic at the cusp $q=0$. The remaining cusps correspond to $q\rightarrow \zeta$ for some root of unity $\zeta$. As $|q|\rightarrow 1$, the function $a_{k_i}^{r_i}(q)$ grows most polynomially: First, ${\bf p}_k^r(\lambda)$ grows at most polynomially in $|\lambda|$. Second, $|{\bf w}_N(\lambda)|\leq 1$ for $N=2,3,4,6$ by the main result of \cite{fomin}. By the analogue of Proposition 3 of \cite{eo2}, there is also a polynomial bound on $|{\bf w}_N(\lambda)|$ for all $N$. A pole of $a_{k_i}^{r_i}(q)$ at some cusp would cause exponential growth of $\langle \prod {\bf p}_{k_i}^{r_i}\rangle_{{\bf w}_N}$ so is excluded. Hence:

\begin{proposition} For all $N\geq 2$, the coefficient $a_{k_i}^{r_i}(q^{1/N})$ is an element of weight $\sum k_i+\#\{i\,\big{|}\, r_i=0\}$ in $QM(\Gamma_1(N))$ over $\Q(\zeta_N)$.
\end{proposition}

Theorem \ref{modular1} follows immediately, because the monomials in the ${\bf p}_k^r$ generate $\Lambda_N$ over $\Q(\zeta_N)$ and $\langle\, \cdot\, \rangle_{{\bf w}_N}$ is linear.

\section{Moduli of cubic, quartic, and sextic differentials}\label{volumesec}
 
 \subsection{Strata of higher differentials} Throughout this section, let $\mu$ be a partition of $N(2g-2)$ with parts $$\mu_i\in \{-N+1,\,\dots,\,-1\}\cup\N.$$
 
 \begin{definition} Let $\mathcal{H}_N(\mu)$ denote the moduli space of pairs $(\Sigma,\omega)$ where $\Sigma$ is a compact Riemann surface and $\omega$ is a non-vanishing meromorphic section of $K^{\otimes N}$ such that $${\rm div}(\omega)=\sum_{i=1}^{\ell(\mu)} \mu_i x_i$$ for some unmarked points $x_i\in\Sigma$. \end{definition} Then $\mathcal{H}_N(\mu)$ is a complex orbifold with period coordinates which we define now. Consider the cyclic cover $\pi\,:\,\widetilde{\Sigma}\rightarrow \Sigma$ which trivializes the monodromy of the flat structure on $\Sigma\backslash\{x_i\}$. Assume for now that $\omega$ is not some power of a lower order differential, so that the degree of the covering is $N$. There is a Galois action of $\langle \zeta_N\rangle$ on $\widetilde{\Sigma}$ by deck transformations, and an abelian differential $\alpha$ such that $\alpha^N=\pi^*\omega$. Furthermore, the pullback of $\alpha$ under the action of $\zeta_N$ is $\zeta_N\alpha$. Let $$V:=H_1(\widetilde{\Sigma},\pi^{-1}\{x_i\};\C)^{\zeta_N}$$ denote the $\zeta_N$-eigenspace of the action of the generator $\zeta_N$ of the deck transformations. We note that $V$ is defined over $\Q(\zeta_N)$, as any representation of $\langle \zeta_N\rangle$ over $\Q$ splits into eigenspaces over this field. 

By \cite{bcggm}, Corollary 2.3, it is known that the periods of $\alpha$ against a basis of $V$ form a local coordinate chart on $\mathcal{H}_N(\mu)$. Thus, we have a natural period map from the moduli space into $V^*$ sending $$(\Sigma,\omega)\mapsto [\alpha]\in V^*.$$ 

\subsection{Tilings and periods} We now restrict to the cases $N=3,4,6$. Note that $V\oplus \overline{V}$ is defined over $\Q$ because $[\Q(\zeta_N):\Q]=2$, and thus has a natural $\Z$-Hodge structure whose integral lattice is $$(V\oplus \overline{V})\cap H_1(\widetilde{\Sigma},\pi^{-1}\{x_i\};\Z).$$ Dualizing endows $V^*\oplus \overline{V}^*$ with a $\Z$-Hodge structure. Let $L\subset V^*$ denote the projection of the lattice to the $V^*$ factor. Define $\theta_N := \zeta_N-\overline{\zeta}_N$.

\begin{proposition}\label{monotile} Let $N=3,4,6$. An element $(\Sigma,\omega)\in\mathcal{H}_N(\mu)$ admits a tiling of its flat structure into fixed size bicolored hexagons, squares, or triangles, respectively, if and only if the period point of $(\Sigma,\omega)$ lies in $L$. \end{proposition}

\begin{proof} Assume $\mu\neq \emptyset$. We first claim that $(\Sigma,\omega)$ admits a tiling iff \begin{align}\label{lattice} \int_\gamma \omega^{1/N}\in \theta_N^{-1}\Z[\zeta_N]\end{align} for all paths $\gamma$ connecting some possibly equal points $x_i$ and $x_j$. If (\ref{lattice}) holds, the flat structure on $\Sigma\backslash \{x_i\}$ admits a reduction of structure group to the crystallographic group $G=\langle \zeta_N\rangle \ltimes \theta_N^{-1}\Z[\zeta_N]$. In particular, the monodromy is valued in $G$. Taking the developing map from the universal cover of $\Sigma\backslash \{x_i\}\rightarrow \C$, we can pull back the tiling on $\C$ which $G$ preserves to the universal cover. This tiling then descends and extends to $\Sigma$. Conversely, a surface tiled by appropriately sized tiles must have periods in this lattice.

Next, we claim that $\alpha\in L$ iff (\ref{lattice}) holds. 
Suppose that $\alpha\in L$. That is, $$\alpha\in H^1(\widetilde{\Sigma},\pi^{-1}\{x_i\};\C)^{\zeta_N}=V^*$$ is the projection of an integral point to $V^*$. Equivalently $\alpha+\overline{\alpha}$ is integral. Applying the action of $\zeta_N$, we also have integrality of $\zeta_N\alpha+\overline{\zeta_N\alpha}$. Solving a linear system, we conclude that $\theta_N\alpha$ is defined over $\Z[\zeta_N]$. In particular, $$\int_\gamma \alpha\in \theta_N^{-1}\Z[\zeta_N]$$ for all $\gamma\in H_1(\widetilde{\Sigma},\pi^{-1}\{x_i\};\Z)$. Any path connecting $x_i$ to $x_j$ on $\Sigma$ lifts to a path connecting a point of $\pi^{-1}\{x_i\}$ to a point of $\pi^{-1}\{x_j\}$ on $\widetilde{\Sigma}$. Thus, (\ref{lattice}) holds, since we may compute any such integral on $\widetilde{\Sigma}$. 

Conversely, suppose (\ref{lattice}) holds. Given any $\gamma\in H_1(\widetilde{\Sigma},\pi^{-1}\{x_i\};\Z)$, we may represent it as a path which terminates on the set $\pi^{-1}\{x_i\}$. So we may compute $\int_\gamma \alpha$ downstairs as the integral of $\omega^{1/N}$ between some points $x_i$ and $x_j$. Hence, for all $\gamma\in H^1(\widetilde{\Sigma},\pi^{-1}\{x_i\};\Z)$, $$\int_\gamma \alpha+\overline{\alpha}\in 2\textrm{Re}\,\theta_N^{-1}\Z[\zeta_N]=\Z$$ which implies $\alpha\in L$. When $\mu=\emptyset$, we have that $(\Sigma,\omega)$ is a flat torus, in which case, the proposition is trivial. \end{proof}

\begin{remark} Suppose $N=6$ and $\mu$ has all even parts. One might expect, as above, that hexagon-tiled surfaces form a lattice in the moduli space $\mathcal{H}_6(\mu)$, but some subtleties arise. The difference with the above cases is that the vertices of the hexagonal tiling of $\C$ do not form a lattice, rather the centers of the hexagons do.

Observe that a hexagon-tiled surface must also be tileable by equilateral triangles, thus it is natural to first impose (\ref{lattice}). The existence of a compatible tiling by hexagons is equivalent to the existence of a vertex $v$ of the triangulation satisfying certain conditions: Let $\gamma_i$ denote a path connecting $v$ to $x_i$ and let $\gamma$ denote any closed path based at $v$. If we have \begin{align*}\int_{\gamma_i} \omega^{1/6}  \in \,& \theta_6^{-1}\Z[\zeta_6]- \Z[\zeta_6]\textrm{ for all }i, \\ &\int_{\gamma} \omega^{1/6} \in \Z[\zeta_6],\end{align*} then $(\Sigma,\omega)$ admits a tiling into hexagons. To construct this tiling, we declare that $w$ is the center of a hexagon if and only if $$\int_v^w \omega^{1/6}\in \Z[\zeta_6]$$ for any path connecting $v$ to $w$. The first condition ensures that no center of a hexagon lies at a singularity of $\Sigma$. Combined, the two conditions ensure that being a center is independent of the path connecting $v$ to $w$.


\end{remark}

\begin{remark}\label{multiple} Suppose that $\delta=\gcd(\mu_1,\dots,\mu_{\ell(\mu)},N)\neq 1$. For each $\delta'\big{|}\delta$, there are some connected components of $\mathcal{H}_N(\mu)$ corresponding to $N$-ic differentials which are a power of an $N\delta'/\delta$-ic differential and when $\delta'\neq \delta$ the tileable surfaces in this connected component admit a more rigid tiling. For instance, for $N=6$ and $\delta=2$ the hexagon-tiled surfaces form a discrete subset of the moduli space, and when $\delta'=1$ these tilings can be bicolored. Thus, it is possible to compute via inclusion-exclusion the number of tiled surfaces which lie on the primitive components, i.e. those where $\delta'=\delta$. \end{remark}

\subsection{The Masur-Veech volume} There is a canonical choice of volume form on $V^*$, the Lebesgue measure, scaled so that $L$ has covolume $1$.
But $\mathcal{H}_N(\mu)$ has infinite volume with respect to this measure because any chart into $V^*$ extends linearly to a cone---when we scale $\omega\mapsto c^N\omega$ for some $c\in \C$, the resulting period point $[\alpha]$ scales by $c$. There is also a Hermitian form on $V^*$ defined by its taking the value $$\textrm{Area}(\Sigma,\omega)=\int_\Sigma |\omega |^{2/N}= \frac{1}{N}\int_{\widetilde{\Sigma}}|\alpha|^2$$ on $(\Sigma,\omega)$. The image of the period map lands in the vectors of positive norm in $V^*$. To get finite volume, one may define $$\mathcal{H}_N^1(\mu):=\{(\Sigma,\omega)\in \mathcal{H}_N(\mu)\,\big{|}\, \textrm{Area}(\Sigma,\omega)=1\},$$ which kills the scaling by an element of $\R^+$. Then one may define the volume of a subset $U\subset \mathcal{H}_N^1(\mu)$ to be $$\textrm{Vol}(U):=\textrm{Vol}_{\rm Lebesgue}\{t \varphi(u)\,\big{|}\,u\in U,\,t\in [0,1]\}$$ where $\varphi$ is the period map to $V^*$. We call this volume the {\it Masur-Veech volume} in analogy with the cases $N=1,2$ see \cite{masur,veech}. We cite recent result due to Duc-Manh Nguyen \cite{ng} which proves:

\begin{theorem}[Nguyen \cite{ng}]\label{duc} The Masur-Veech volume ${\rm Vol}(\mathcal{H}^1_N(\mu))$ is finite for all $N$ and $\mu$ satisfying $\mu_i>-N$. \end{theorem}

When $N=1$, Kontsevich and Zorich conjectured, and Eskin and Okounkov proved in \cite{eo1}, that $\pi^{-2g}\textrm{Vol}(\mathcal{H}^1_1(\mu))\in\Q.$ If $g=0$ and $N=2$, Theorem 1.1 of \cite{aez} implies that $\pi^{-2g_{\textrm{eff}}}\textrm{Vol}(\mathcal{H}^1_2(\mu))\in\Q$ where $g_{\textrm{eff}}:=\hat{g}-g$ and $\hat{g}$ is the genus of the double cover $\widetilde{\Sigma}$. A proof for all $g>0$ of the same result has not yet appeared in the literature. Using the main results of this paper, we may extend a weaker form of rationality to the cases $N=3,4,6$:

\begin{theorem}\label{volume} Suppose $N=3,4,6$. Then $$(i\theta_N)^{-\dim \mathcal{H}_N}{\rm Vol}(\mathcal{H}_N^1(\mu))\in\Q[2\pi i\theta_N].$$ \end{theorem}

\begin{proof}[Sketch.] The lattice points in $L$ evenly sample the Lebesgue measure, and therefore we can extract the volume using the asymptotics of the number of tiled surfaces, see Proposition 1.6 and Proposition 3.2 of \cite{eo1}, or Proposition 7 of \cite{goujard}. When $N=3,4,6$ (similar formulas hold for $N=1,2$) \begin{align*}  {\rm Vol}(\mathcal{H}^1_N(\mu))&=\lim_{q\rightarrow 1} \frac{ H_N^{\circ}(\emptyset\,|\,\emptyset,\emptyset, \mu+N)(1-q)^{\dim}}{(\textrm{Area of Tile})^{\dim}\dim! }
\end{align*} where $\dim=\dim_\C\mathcal{H}_N^1(\mu)$ and $$\textrm{Area of Tile}=
\left\{
	\begin{array}{llll}
		\frac{1}{4\sqrt{3}} & \mbox{if } N=6,\,\textrm{Tile = triangle}   \\
		\frac{1}{4} & \mbox{if } N=4,\,\textrm{Tile = square} \\
		\frac{1}{2\sqrt{3}} & \mbox{if } N=3,\,\textrm{Tile = vertex-bicolored hexagon} \\
		\frac{1}{2} & \mbox{if } N=2,\,\textrm{Tile = edge-bicolored quadrilateral} \\
		1 & \mbox{if } N=1,\,\textrm{Tile = asymmetric quadrilateral} \\		
	\end{array}
\right.$$ is the area of the fundamental tile for a surface with period point in $L$. For $N=6$ and $\mu$ having even parts, a similar formula holds, but there is a rational constant determined by the index of the hexagon-tiled surfaces within the triangulated surfaces. If there are components of a stratum for which every surface $(\Sigma,\omega)$ in that component has a non-trivial automorphism, the above formula is valid only when we weight the volume of that component by the appropriate factor.

When $N=3,4,6$, the rings of quasimodular forms for $\Gamma_1(N)$ are: \begin{align*}
QM^*(\Gamma_1(3))&=\Q[\theta_3 E_1^1 , E_2, \theta_3 E_3^1] \\
QM^*(\Gamma_1(4))&=\Q[\theta_4 E_1^2, E_2, E_2^1] \\
QM^*(\Gamma_1(6))&=\Q[\theta_6 E_1^1, \theta_6 E_1^2, E_2] \end{align*}

The factors of $\theta_N$ are included so that the coefficients are rational. More precisely, by Theorem 4.2.3 of \cite{diashur}, we have for all $k\geq 3$ the Fourier expansion $$E_k^r(q)=(k-1)!\!\!\sum_{d\equiv r\,(N)}\!\!\!\!'\,\left(\frac{-N}{2\pi id}\right)^k+\sum_{n\geq 1} q^n\sum_{m\mid n}\textrm{sgn}(m)m^{k-1}\zeta_N^{rm}$$ and thus the Fourier coefficients lie in $\Q(\zeta_N)\cap \R=\Q$ when $k$ is even and lie in $\Q(\zeta_N)\cap i\R= \Q\cdot \theta_N$ when $k$ is odd. Note that when $k=1$, we must add an additional constant to make $E_k^r(q)$ modular, and when $k=2$ the above series is only quasimodular.

Let $E=E_k^r$ be one of the above Eisenstein series, not equal to $E_2$. Then $$\tau^{-wt}E(-1/\tau)=c_0+c_1q^{1/N}+c_2q^{2/N}+\cdots$$ admits a holomorphic Fourier series expansion because the width of the cusp $\tau=0$ on the modular curve $X_1(N)$ is equal to $N$. In fact, the result is an Eisenstein series for $\Gamma^1(N)$ with $c_i\in\Q$. Taking the limit of $E(q)$ as $q\rightarrow 1$ is the same as taking the limit of $E(\tau)$ as $\tau\rightarrow 0$. In the expansion $$E(\tau)=(-1/\tau)^{wt}\left[c_0+c_1e^{-2\pi i/N\tau}+c_2e^{-4\pi i/N\tau}+\cdots \right]$$ all of the terms decay exponentially as $\tau\rightarrow 0$ except the contribution from $c_0$. For $E_2$ we have the quasi-modular transformation $$E_2(\tau)= (-1/\tau)^2\left[-\frac{1}{24}+e^{-2\pi i/\tau}+3e^{-4\pi i/\tau}+\cdots \right]+\frac{1}{4\pi i\tau}.$$ Setting $\ell=-(2\pi i\tau)^{-1}$, we may make the substitutions \begin{align} \begin{aligned}\label{asymptotics} \theta_N^{wt} E(q)&\longleftrightarrow c_0(2\pi i\theta_N \ell)^{wt} \\ E_2(q)&\longleftrightarrow \tfrac{1}{6}\pi^2\ell^2-\tfrac{1}{2}\ell \end{aligned} \end{align} and let $\ell\rightarrow \infty$ to analyze the asymptotic behavior of the above generating functions. Thus the volumes in the statement of the theorem are expressible as the limit $$\frac{\lim_{\ell \rightarrow \infty}\, (1-e^{-1/\ell})^{\dim}P(\ell)}{(\textrm{Area of Tile})^{\dim}\dim !}$$ where $P(\ell)$ is by definition the polynomial gotten by expressing the quasimodular form $H_N^{\circ}(\emptyset\,|\,\emptyset,\emptyset, \mu+N)$ in terms of the generators $E(q)$ and $E_2(q)$, then making the substitutions dictated by (\ref{asymptotics}). The volume is non-zero and by Theorem \ref{duc} is finite. Hence the degree of $P(\ell)$ is exactly $\dim$, and the limit is the leading coefficient of $P(\ell).$ Thus, the limit lies in $\Q[2\pi i\theta_N]$. \end{proof}

\begin{remark} In \cite{cmz}, the polynomial $P(\ell)$ is called the {\it asymptotic substitution} of the corresponding quasimodular form. \end{remark} 

\subsection{Conjectures on the cases $N=5,\,N\geq 7$}\label{conjs} A surprising result of Proposition \ref{hooks} is that after scaling by an appropriate constant, $$\left(\frac{\dim \lambda}{|\lambda|}\right)^2\prod_t {\bf c}_t(\lambda)$$ for $t$ ranging in the sets $\{2,2,2,2\}$, $\{3,3,3\}$, $\{2,4,4\}$, $\{2,3,6\}$ admits a generalization beyond the values $N=2$, $3$, $4$, $6$, respectively, to all positive integers $N$. Once this generalization is made, Theorem \ref{conj} and its proof are valid for all $N$. Furthermore, the resulting brackets $\langle F \rangle_{{\bf w}_{N,\eta}}$ are modular forms for $\Gamma(N)$ with coefficients in $\Q$. {\it If the elliptic orbifold of order $5$ existed}, it seems likely that these brackets would encode its Hurwitz theory. This raises the following question:

\begin{question}\label{firstq} Do the brackets $\langle {\bf F} \rangle_{{\bf w}_N}$ for $N=5, N\geq 7$ have an enumerative interpretation? \end{question}

Note that the shifted symmetric function ${\bf F}_\eta^D(\lambda)\in \Lambda^*_N$ appearing in Proposition \ref{modify} for $D=\{\emptyset\,\big{|}\,\emptyset,\emptyset,\mu\}$ admits a natural generalization ${\bf F}_\eta^\mu(\lambda)$ to any $N$-core $\eta$. Consider $$H_N^\bullet(\mu,q):=\sum_{N\textrm{-cores }\eta} \langle {\bf F}^\mu_\eta(\lambda)\rangle_{{\bf w}_{N,\eta}}.$$ One can formally produce the ``connected" generating function $H_N^\circ(\mu,q)$ via inclusion-exclusion on the components of $\mu$. This quasimodular form is the generating function for tilings in the stratum $\mathcal{H}_N(\mu-N)$ whenever $N=1,2,3,4,6$ but can it be related to the stratum for all $N$? 

\begin{conjecture} Fix a stratum $\mathcal{H}_N(\mu-N)$ and denote its dimension ${\rm dim}$. Let $P(\ell)$ be the asymptotic polynomial of $H_N^\circ(\mu,q)$. Choose a constant $A_N$ and take the limit $$\lim_{\ell\rightarrow \infty}\frac{(1-e^{-1/\ell})^{\rm dim} P(\ell)}{(A_N)^{\rm dim}\,{\rm dim}!}.$$ Does this limit exist and if so is there a fixed value of $A_N$ for which the limit equals ${\rm Vol}(\mathcal{H}^1_N(\mu-N))$? \end{conjecture}

\appendix

\section{Numerical Example}

In this example, we compute the generating function of vertex-bicolored hexagon tilings of the sphere with curvature profile $\kappa=(2,2,1,1)$ over the black vertices and no curvature over the white vertices. These tilings cover $\mathbb{P}_{3,3,3}$ with ramification profile $3^d$, $3^d$, and $(3^{d-2},2,2,1,1)$. The generating function for possibly disconnected covers with no orbifold-unramified components and this ramification profile is $$\langle {\bf g}^3_{2,2,1,1}\rangle_{{\bf w}_3}$$ where ${\bf g}^3_{2,2,1,1}$ is the unique element of $\Lambda_3$ which takes the value $$\frac{{\bf f}_{3,\dots,3,2,2,1,1}(\lambda)}{{\bf f}_{3,\dots,3}(\lambda)}$$ on any $3$-decomposable partition $\lambda$. No strict subset of the curvatures form the profile of a tiling, hence the bracket already counts connected covers. The weight of ${\bf g}_{2,2,1,1}^3$ will be $(2+2+1+1)/3=2$, similar to the formula in \cite{eo2}, though we did not prove this and the most naive bound from Section \ref{shiftedsymm} on the weight would be $4$. We may represent it as a linear combination \begin{align*} {\bf g}_{2,2,1,1}^3=&A_1({\bf p}_1^1)^2+A_2({\bf p}_1^1{\bf p}_1^2)+A_3({\bf p}_1^2)^2+A_4({\bf p}_2^1)\,+  \\ & A_5({\bf p}_2^2)+A_6({\bf p}_1^0)+A_7({\bf p}_1^1)+A_8({\bf p}_1^2)+A_9\end{align*} for constants $A_i\in\Q(\zeta_3)$. Plugging in enough partitions $\lambda$ on both sides, we determine the constants:
\begin{align*}
&A_1=\frac{1}{81}\zeta_6 && A_2=-\frac{1}{24} && A_3=\frac{1}{81}\zeta_6^{-1} \\
&A_4=\frac{1}{32}(1+\zeta_6) && A_5=\frac{1}{32}(1+\zeta_6^{-1}) && A_6=0 \\
&A_7=\frac{1}{12}\zeta_3 && A_8=\frac{1}{12}\zeta_3^{-1} && A_9=\frac{1}{18}.
\end{align*}

This formula allows us to easily compute a relatively large number of coefficients of $\langle {\bf g}^3_{2,2,1,1}\rangle_{{\bf w}_3}$, as we only need a table of $3$-decomposable partitions, along with their weights ${\bf w}_3(\lambda)$, which are easily computed by Proposition \ref{hooks}. This average is, by Theorem \ref{modular1}, the result of the substitution $q\mapsto q^3$ into a quasimodular form for $\Gamma_1(3)$ of weight bounded by $2$. Thus, \begin{align*}H(q):=H^\bullet_3(\emptyset\,\big{|}\,\emptyset,\emptyset,(-1,&-1,-2,-2))=B_1+B_2X+B_3Y+B_4Z\hspace{2pt}\textrm{ where} \\ X&=\tfrac{1}{6}+\sum \sigma_0^\chi(n)q^n \\ Y&=-\tfrac{1}{24}+\sum \sigma_1(n)q^n  \\ Z&=-\tfrac{1}{24}+\sum \sigma_1(n)q^{3n}. \end{align*} Here $\sigma_0^\chi(n)=\sum_{d\mid n,\,d> 0} \chi(d)$ where $\chi$ is the unique nontrivial Dirichlet character mod $3$. Either using the formula for ${\bf g}_{2,2,1,1}$ in terms of ${\bf p}_{k_i}^{r_i}$ or by direct computation, we determine that $$H(q) =\frac{1}{2}q^2+q^3+\cdots$$ from which we can solve $B_1=1/18$, $B_2=-1/6$, $B_3=1/6$, and $B_4=1/2$. We conclude the formula $$\left(\!
		\begin{array}{ll}
			\,\textrm{weighted }\#\,3\textrm{-duck tilings with }n\textrm{ tiles} \\
			\textrm{and curvature }(\kappa_b,\kappa_w)=((2,2,1,1),\emptyset)
		\end{array}\!\!
	\right)=-\tfrac{1}{6}\sigma_0^\chi(n)+\tfrac{1}{6}\sigma_1(n)+\tfrac{1}{2}\sigma_1(n/3).$$

For instance, when $n=3$, there is a single tiling with no nontrivial orientation-preserving automorphisms:

\begin{figure}[H]
\includegraphics[width=3in]{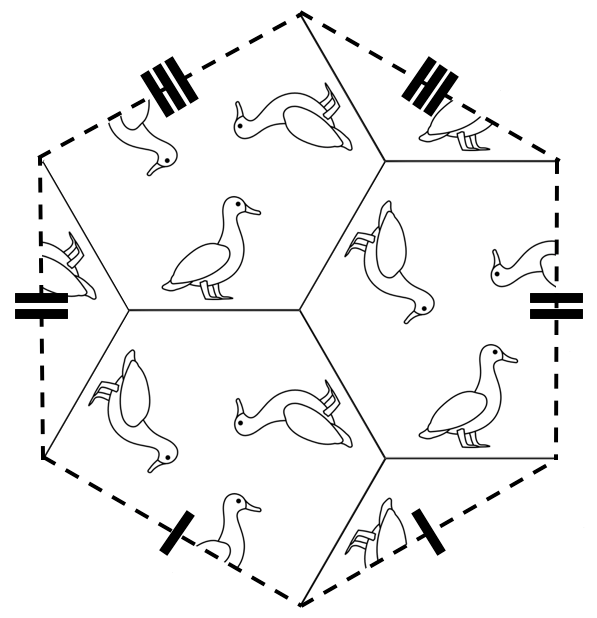}
\caption{A tiling with curvatures $((2,2,1,1),\emptyset)$.}
\label{3tiles}
\end{figure}

Finally, we compute the volume of $\mathcal{H}_3^1(2,2,1,1)$ from the generating function for tilings. In the notation of Theorem \ref{volume}, the leading term of $P(\ell)$ is $\frac{\pi^2}{9}\ell^2$. Since the moduli space has dimension $2$, we conclude that $$\textrm{Vol}\,\mathcal{H}_3^1(2,2,1,1)=\frac{2\pi^2}{3}.$$ The projectivization of the moduli space is a further quotient by the action of $S^1$ and thus has volume $\textrm{Vol}\,\mathbb{P}\mathcal{H}_3(2,2,1,1)=\frac{\pi}{3}.$ This agrees with the Gauss-Bonnet formula in Theorem 1.2 of \cite{mcmullen} for the complex-hyperbolic volume of the moduli space.

\end{document}